\documentclass[a4papet,10pt]{article}
\usepackage{amsfonts,amssymb,mathtools}
\usepackage{graphicx}
\usepackage{amsmath}
\usepackage[latin1]{inputenc} 
\usepackage{bbm}
\usepackage{pstricks}
\usepackage{psfrag}
\usepackage{appendix}
\usepackage{amsthm}
\usepackage{dsfont}
\usepackage{stackrel}
\usepackage{eufrak}
\usepackage{mathrsfs}
\usepackage{enumerate}
\usepackage{cite}
\usepackage{chngcntr}
\counterwithin{figure}{section}
\usepackage[a4paper,top=30mm,bottom=30mm,left=35mm,right=35mm]{geometry}
\numberwithin{equation}{section}

\newcommand{\prob}{\mathbb{P}}
\newcommand{\Ex}{\mathbb{E}}
\newcommand{\Rl}{\mathbb{R}}

\newcommand{\dom}{\mbox{dom}\,}

\newcommand{\cl}{\mbox{cl}\,}

\newcommand{\Is}{I_{\text{s}}}
\newcommand{\Lambdas}{\Lambda_{\text{s}}}
\newcommand{\Ii}{I_{\text{i}}}
\newcommand{\IC}{I_{\text{C}}}

\newcommand{\LC}{\Upsilon_{\text{C}}}
\newcommand{\ells}{\ell_{\text{s}}}

\newcommand{\elli}{\ell_{\text{i}}}
\newcommand{\rew}{\mathcal{X}}
\newcommand{\Brew}{\mathcal{B}(\mathcal{X})}

\newtheorem*{cramer}{Cram\'er's theorem}

\newtheorem{assumption}{Assumption}[section]
\newtheorem{theorem}{Theorem}[section]
\newtheorem{proposition}{Proposition}[section]
\newtheorem{lemma}{Lemma}[section]

\newtheorem{example}{Example}[section]

\title{Large Deviations in Discrete-Time Renewal Theory}

  \author{Marco Zamparo\footnote{Dipartimento Scienza Applicata e Tecnologia, Politecnico di Torino,
  Corso Duca degli Abruzzi 24, Torino, I-10129, Italy
\newline \phantom{aaz} E-mail: \texttt{marco.zamparo@uniba.it}}}
  
\date{}

\begin{document}  
\maketitle

\begin{abstract}
We establish sharp large deviation principles for cumulative rewards
associated with a discrete-time renewal model, supposing that each
renewal involves a broad-sense reward taking values in a real
separable Banach space.  The framework we consider is the pinning
model of polymers, which amounts to a Gibbs change of measure of a
classical renewal process and includes it as a special case.  We first
tackle the problem in a constrained pinning model, where one of the
renewals occurs at a given time, by an argument based on convexity and
super-additivity. We then transfer the results to the original pinning
model by resorting to conditioning.\\


\end{abstract}

\section{Introduction}

\subsection{Renewals and Cram\'er's theorem}
\label{sec:basics}

{\it Renewal models} are widespread tools of probability, finding
application in Queueing Theory \cite{AsmussenBook}, Insurance
\cite{DicksonBook}, and Finance \cite{RSSTBook} among others.  A
renewal model describes some event that occurs at the {\it renewal
  times} $T_1,T_2,\ldots$ involving the {\it rewards} $X_1,X_2,\ldots$
respectively. If $S_1,S_2,\ldots$ are the {\it waiting times} for a
new occurrence of the event, then the renewal time $T_i$ can be
expressed for each $i\ge 1$ in terms of waiting times as
$T_i=S_1+\cdots+S_i$. This paper deals with cumulative rewards in
renewal models with waiting times taking discrete values and rewards
taking values in a Banach space. Specifically, we assume that the
waiting time and reward pairs $(S_1,X_1),(S_2,X_2),\ldots$ form an
independent and identically distributed sequence of random variables
on a probability space $(\Omega,\mathcal{F},\prob)$, the waiting times
being valued in $\{1,2,\ldots\}\cup\{\infty\}$ and the rewards being
valued in a real separable Banach space $(\rew,\|{\cdot}\|)$ equipped
with the Borel $\sigma$-field $\Brew$. Any dependence between $X_i$
and $S_i$ is allowed.  The {\it cumulative reward} by the integer time
$t\ge 0$ is the random variable $W_t:=\sum_{i\ge
  1}X_i\mathds{1}_{\{T_i\le t\}}$, which is measurable because $\rew$
is separable \cite{Talagrand}.  The stochastic process $t\mapsto W_t$
is the so-called {\it renewal-reward process} or {\it compound renewal
  process}, which plays an important role in applications
\cite{AsmussenBook,DicksonBook,RSSTBook}. The strong law of large
numbers can be proved for a renewal-reward process under the optimal
hypotheses $\Ex[S_1]<+\infty$ and $\Ex[\|X_1\|]<+\infty$, $\Ex$
denoting expectation with respect to the law $\prob$, by combining the
argument of renewal theory \cite{AsmussenBook} with the classical
strong law of large numbers of Kolmogorov in separable Banach spaces
\cite{Talagrand}.


In this paper we characterize the large fluctuations of the cumulative
reward $W_t$ by establishing large deviation principles that
generalize the Cram\'er's theorem to discrete-time renewal models.
Cram\'er's theorem describes the large fluctuations of non-random sums
of random variables, such as the total reward versus the number of
renewals $n$ given by $\sum_{i=1}^n X_i$.  It involves the {\it rate
  function} $\IC$ that maps each point $w\in\rew$ in the extended real
number
$\IC(w):=\sup_{\varphi\in\rew^\star}\{\varphi(w)-\ln\Ex[e^{\varphi(X_1)}]\}$,
where $\rew^\star$ is the topological dual of $\rew$.  The following
sharp form of Cram\'er's theorem has been obtained by Bahadur and
Zabell \cite{BaZa} through an argument based on convexity and
sub-additivity.
\begin{cramer}
  The following conclusions hold:
  \begin{enumerate}[{\upshape(a)}]
  \item the function $\IC$ is lower semicontinuous and proper convex;
  \item if $G\subseteq\rew$ is open, then
    \begin{equation*}
      \liminf_{n\uparrow\infty}\frac{1}{n}\ln\prob\Bigg[\frac{1}{n}\sum_{i=1}^n X_i\in G\Bigg]\ge -\inf_{w\in G}\{\IC(w)\};
\end{equation*}
  \item if $F\subseteq\rew$ is compact, open convex, or closed convex,
    then
\begin{equation*}
      \limsup_{n\uparrow\infty}\frac{1}{n}\ln\prob\Bigg[\frac{1}{n}\sum_{i=1}^n X_i\in F \Bigg]\le-\inf_{w\in F}\{\IC(w)\}.
\end{equation*}
Furthermore, if $\rew$ is finite-dimensional, then this bound is valid
for any closed set $F$ provided that $\Ex[e^{\xi\|X_1\|}]<+\infty$ for
  some number $\xi>0$.
\end{enumerate}
\end{cramer}
Earlier, Donsker and Varadhan \cite{Donsker} proved Cram\'er's theorem
under the stringent exponential moment condition
$\Ex[e^{\xi\|X_1\|}]<+\infty$ for all $\xi>0$. Importantly, they
showed that under this condition the upper bound in part (c) holds for
any closed set $F$ even when $\rew$ is infinite-dimensional.

Along with the use as stochastic processes, discrete-time renewal
models find application in Equilibrium Statistical Physics with a
different interpretation of the time coordinate. In particular, they
are employed in studying the phenomenon of polymer pinning, whereby a
polymer consisting of $t\ge 1$ monomers is pinned by a substrate at
the monomers $T_1,T_2,\ldots$ that represent renewed events along the
polymer chain \cite{Giac,Hollander_Polymer}. Supposing that the
monomer $T_i$ contributes an energy $-v(S_i)$ provided that $T_i\le
t$, $v$ being a real function over $\{1,2,\ldots\}\cup\{\infty\}$
called the {\it potential}, the state of the polymer is described by
the perturbed law $\prob_t$ defined on the measurable space
$(\Omega,\mathcal{F})$ by the Gibbs change of measure
\begin{equation*}
\frac{d\prob_t}{d\prob}:=\frac{e^{H_t}}{Z_t},
\end{equation*}
where $H_t:=\sum_{i\ge 1}v(S_i)\mathds{1}_{\{T_i\le t\}}$ is the {\it
  Hamiltonian} and the normalizing constant $Z_t:=\Ex[e^{H_t}]$ is the
{\it partition function}. The model $(\Omega,\mathcal{F},\prob_t)$ is
called the {\it pinning model} (PM) and generalizes the original
renewal model corresponding to the potential $v=0$.  The theory of
large deviations we develop in this paper is framed within the PM
supplied with the hypotheses of aperiodicity and extensivity.  The
{\it waiting time distribution} $p:=\prob[S_1=\cdot\,]$ is said to be
{\it aperiodic} if $\prob[S_1<\infty]>0$ and there is no proper
sublattice of $\{1,2,\ldots\}$ containing the support of $p$.  We
point out that a generic $p$ can be made aperiodic when
$\prob[S_1<\infty]>0$ by simply changing the time unit.
\begin{assumption}
\label{ass1}
The waiting time distribution $p$ is aperiodic.
\end{assumption}
We say that the potential $v$ is {\it extensive} if there exists a
real number $z_o$ such that $e^{v(s)}p(s)\le e^{z_os}$ for all
$s$. For instance, any potential $v$ with the property that
$\sup_{s\ge 1}\{v(s)/s\}<+\infty$ is extensive. Extensive potentials
are the only that serve Equilibrium Statistical Physics, where the
partition function
$Z_t\ge\Ex[e^{H_t}\mathds{1}_{\{S_1=t\}}]=e^{v(t)}p(t)$ is expected to
grow exponentially in $t$ in order to define the free energy
\cite{Giac,Hollander_Polymer}.
\begin{assumption}
\label{ass2}
The potential $v$ is extensive.
\end{assumption}

Together with the PM we consider the {\it constrained pinning model}
(CPM) where the last monomer is always pinned by the substrate
\cite{Giac,Hollander_Polymer}. It corresponds to the law $\prob_t^c$
defined on the measurable space $(\Omega,\mathcal{F})$ through the
change of measure
\begin{equation*}
\frac{d\prob_t^c}{d\prob}:=\frac{U_t e^{H_t}}{Z_t^c},
\end{equation*}
where $U_t:=\sum_{i\ge 1}\mathds{1}_{\{T_i=t\}}$ is the renewal
indicator, which takes value 1 if $t$ is a renewal and value 0
otherwise, and $Z_t^c:=\Ex[U_te^{H_t}]$ is the partition function.
Our interest in the CPM is twofold.  On the one hand, it turns out to
be an effective mathematical tool to tackle the PM. Indeed, we can
obtain a large deviation principle within the CPM by an argument based
on convexity and super-additivity, and then transfer it to the PM by
conditioning.  The mentioned argument is a generalization of the
approach to Cram\'er's theorem by Bahadur and Zabell \cite{BaZa},
which in turn can be traced back to the method of Ruelle \cite{Ruelle}
and Lanford \cite{Lanford} for proving the existence of various
thermodynamic limits. On the other hand, the CPM is a significant
framework in itself because it is the mathematical skeleton of the
Poland-Scheraga model of DNA denaturation and of some relevant lattice
models of Statistical Mechanics, as discussed by the author in
Ref.\ \cite{Models} where use of the theory developed in the present
paper is made. These models are the cluster model of fluids proposed
by Fisher and Felderhof, the model of protein folding introduced
independently by Wako and Sait\^o first and Mu\~noz and Eaton later,
and the model of strained epitaxy considered by Tokar and Dreyss\'e.
The macroscopic observables that enter the thermodynamic description
of these systems turn out to be cumulative rewards corresponding to
rewards of the order of magnitude of the waiting times \cite{Models}.

Before introducing our main results, we must say that the CPM is not
well-defined a priori. In fact, it may happen with full probability
that the time $t$ is not a renewal, so that $Z_t^c=0$. However,
assumption \ref{ass1} resulting in $Z_t^c>0$ for every sufficiently
large $t$ settles the problem at least for all those $t$. To verify
this fact, we observe that aperiodicity of $p$ entails that there
exist $m$ coprime integers $\sigma_1,\ldots,\sigma_m$ such that
$p(\sigma_l)>0$ for each $l$. The bound $Z_t^c\ge
\Ex[U_te^{H_t}\prod_{i=1}^n\mathds{1}_{\{S_i=s_i\}}]=\prod_{i=1}^n
e^{v(s_i)}p(s_i)$ if $t=\sum_{i=1}^ns_i$ yields $Z^c_t>0$ whenever $t$
is an integer conical combination of $\sigma_1,\ldots,\sigma_m$.  On
the other hand, the Frobenius number $t_c\ge 0$ associated with
$\sigma_1,\ldots,\sigma_m$ is finite since these integers are coprime
and by definition any $t>t_c$ can be expressed as an integer conical
combination of them.  It follows that $Z_t^c>0$ for all $t>t_c$.

\subsection{Statement of main results}
\label{par:mainres}

This section reports the main results of the paper. In the sequel,
assumptions \ref{ass1} and \ref{ass2} are tacitly supposed to be
satisfied and the topological dual $\rew^\star$ of $\rew$ is
understood as a Banach space with the norm induced by $\|\cdot\|$.
Let $z$ be the function that maps each linear functional
$\varphi\in\rew^\star$ in the extended real number $z(\varphi)$
defined by
\begin{equation}
z(\varphi):=\inf\bigg\{\zeta\in\Rl\,:\,\Ex\Big[e^{\varphi(X_1)+v(S_1)-\zeta S_1}\mathds{1}_{\{S_1<\infty\}}\Big]\le 1\bigg\},
\label{defzetafun}
\end{equation}
where the infimum over the empty set is customarily interpreted as
$+\infty$. The following proposition puts this function into context
by relating $z$ to the scaled cumulant generating function of $W_t$
within the CPM.  According to this proposition, $z(0)$ turns out to be
the free energy of the CPM \cite{Giac,Hollander_Polymer} and, more in
general, $z(\varphi)$ can be regarded as the free energy of a CPM with
the (possibly non-extensive) potential
$v+\ln\Ex[e^{\varphi(X_1)}|S_1=\cdot\,]$.
\begin{proposition}
  \label{proposition:free_energy}
  The function $z$ is proper convex and lower semicontinuous. The following limit holds
  for every $\varphi\in\rew^\star$:
  \begin{equation*}
  \lim_{t\uparrow\infty}\frac{1}{t}\ln\Ex\big[U_te^{\varphi(W_t)+H_t}\big]=z(\varphi).
\end{equation*}
\end{proposition}
Denoting the expectation with respect to the law $\prob_t^c$ by
$\Ex_t^c$, proposition \ref{proposition:free_energy} entails that
$\lim_{t\uparrow\infty}(1/t)\ln\Ex_t^c[e^{\varphi(W_t)}]=z(\varphi)-z(0)$
for all $\varphi\in\rew^\star$, so that $z-z(0)$ is exactly the scaled
cumulant generating function of $W_t$ within the CPM.  We stress that
the number $z(0)$ is finite. Indeed, we have $\Ex[e^{v(S_1)-\zeta
    S_1}\mathds{1}_{\{S_1<\infty\}}]>1$ for all sufficiently negative
$\zeta$ as $\prob[S_1<\infty]>0$ by assumption \ref{ass1} and, at the
same time, $\Ex[e^{v(S_1)-\zeta
    S_1}\mathds{1}_{\{S_1<\infty\}}]=\sum_{s\ge 1}e^{v(s)-\zeta
  s}p(s)\le 1$ for all $\zeta\ge z_o+\ln2$, $z_o$ being the number
introduced by assumption \ref{ass2}. The function $z$ is finite
everywhere in the following case, which is relevant for Statistical
Mechanics as it comprises the macroscopic observables that enter the
thermodynamic description of the system \cite{Models}.
\begin{example}
\label{zfintutto}
The function $z$ is finite everywhere if the reward $X_1$ is dominated
by the waiting time $S_1$ in the sense that $\|X_1\|\le M S_1$ with
full probability for some constant $M<+\infty$. This follows from the
facts that, for any given $\varphi\in\rew^\star$,
$\Ex[e^{\varphi(X_1)+v(S_1)-\zeta
    S_1}\mathds{1}_{\{S_1<\infty\}}]\ge\Ex[e^{-M\|\varphi\|S_1+v(S_1)-\zeta
    S_1}\mathds{1}_{\{S_1<\infty\}}]>1$ for all sufficiently negative
$\zeta$, as $\prob[S_1<\infty]>0$ by assumption \ref{ass1}, and
$\Ex[e^{\varphi(X_1)+v(S_1)-\zeta
    S_1}\mathds{1}_{\{S_1<\infty\}}]\le\sum_{s\ge
  1}e^{M\|\varphi\|s+v(s)-\zeta s}p(s)\le 1$ for all $\zeta\ge
z_o+M\|\varphi\|+\ln2 $ with $z_o$ given by assumption \ref{ass2}.
\end{example}

We use the function $z$ to construct a rate function.  Let $I$ be the
Fenchel-Legendre transform of $z-z(0)$, which associates every point
$w\in\rew$ with the extended real number $I(w)$ given by
\begin{equation}
I(w):=\sup_{\varphi\in\rew^\star}\Big\{\varphi(w)-z(\varphi)+z(0)\Big\}.
\label{defIorate}
\end{equation}
The following theorem extends Cram\'er's theorem to the cumulative
reward $W_t$ with respect to the CPM and constitutes our first main
result. It is proved together with proposition
\ref{proposition:free_energy} in Section \ref{sec:constrainedmodels}.
\begin{theorem}
\label{mainth1}
  The following conclusions hold:
  \begin{enumerate}[{\upshape(a)}]
  \item the function $I$ is lower semicontinuous and proper convex;
  \item if $G\subseteq\rew$ is open, then
    \begin{equation*}
      \liminf_{t\uparrow\infty}\frac{1}{t}\ln\prob_t^c\bigg[\frac{W_t}{t}\in G\bigg]\ge -\inf_{w\in G}\{I(w)\};
    \end{equation*}
\item if $F\subseteq\rew$ is compact, open convex, closed convex, or
  any convex set in $\Brew$ when $\rew$ is finite-dimensional, then
\begin{equation*}
      \limsup_{t\uparrow\infty}\frac{1}{t}\ln\prob_t^c\bigg[\frac{W_t}{t}\in F\bigg]\le-\inf_{w\in F}\{I(w)\}.
\end{equation*}
Furthermore, if $\rew$ is finite-dimensional, then this bound is valid
for any closed set $F$ provided that $\Ex[e^{\xi\|X_1\|+v(S_1)-\zeta
    S_1}\mathds{1}_{\{S_1<\infty\}}]<+\infty$ for some numbers
$\zeta\ge 0$ and $\xi>0$.
\end{enumerate}
\end{theorem}
The lower bound in part (b) and the upper bound in part (c) are
called, respectively, {\it large deviation lower bound} and {\it large
  deviation upper bound} \cite{DemboBook,Hollander}. When a lower
semicontinuous function $I$ exists so that the large deviation lower
bound holds for each open set $G$ and the large deviation upper bound
holds for each compact set $F$, then $W_t$ is said to satisfy a {\it
  weak large deviation principle} (weak LDP) with {\it rate function}
$I$ \cite{DemboBook,Hollander}. If the large deviation upper bound
holds more generally for every closed set $F$, then $W_t$ is said to
satisfy a {\it full large deviation principle} (full LDP)
\cite{DemboBook,Hollander}.  Theorem \ref{mainth1} states that the
cumulative reward $W_t$ satisfies a weak LDP with rate function $I$
given by (\ref{defIorate}) within the CPM.  If in addition $\rew$ is
finite-dimensional and the exponential moment condition
$\Ex[e^{\xi\|X_1\|+v(S_1)-\zeta
    S_1}\mathds{1}_{\{S_1<\infty\}}]<+\infty$ is fulfilled for some
$\zeta\ge 0$ and $\xi>0$, as certainly occurs in example
\ref{zfintutto} for any $\xi>0$ and $\zeta>M\xi+z_o$, then $W_t$
satisfies a full LDP.  Regarding the validity of a full LDP for
general infinite-dimensional Banach spaces $\rew$, finding sufficient
conditions is a harder problem that will be the focus of future
studies. Trying to sketch an analogy with Cram\'er's theorem and the
work by Donsker and Varadhan \cite{Donsker}, one should probably
investigate situations where there exists $\zeta\ge 0$ such that
$\Ex[e^{\xi\|X_1\|+v(S_1)-\zeta
    S_1}\mathds{1}_{\{S_1<\infty\}}]<+\infty$ for all $\xi>0$.

Let us move now to the PM, where there is no constraint on the last
monomer. At variance with the CPM, the scaled cumulant generating
function of $W_t$ may not exist in the PM, but the following
proposition, which is proved in Section \ref{sec:freemodels}, shows
that at least some bounds hold true. Set
$\elli:=\liminf_{t\uparrow\infty}(1/t)\ln\prob[S_1>t]$ and
$\ells:=\limsup_{t\uparrow\infty}(1/t)\ln\prob[S_1>t]$, and bear in
mind that $-\infty\le\elli\le\ells\le 0$.
\begin{proposition}
\label{proposition:freefree_energy}  
The following bounds hold for all $\varphi\in\rew^\star$:
  \begin{eqnarray}
    \nonumber
    z(\varphi)\vee\elli&\le&\liminf_{t\uparrow\infty}\frac{1}{t}\ln\Ex\big[e^{\varphi(W_t)+H_t}\big]\\
    \nonumber
  &\le&\limsup_{t\uparrow\infty}\frac{1}{t}\ln\Ex\big[e^{\varphi(W_t)+H_t}\big]\le z(\varphi)\vee\ells.
\end{eqnarray}
\end{proposition}
Denoting by $\Ex_t$ the expectation with respect to $\prob_t$,
proposition \ref{proposition:freefree_energy} entails that the limit
$\lim_{t\uparrow\infty}(1/t)\ln\Ex_t[e^{\varphi(W_t)}]$ exists, and
equals $z(\varphi)\vee\ells-z(0)\vee\ells$, if either $\elli=\ells$ or
$z(\varphi)\ge\ells$. Thus, the scaled cumulant generating function of
$W_t$ with respect to the PM is defined if either $\elli=\ells$, which
includes the case $\ells=-\infty$, or the condition
$z(\varphi)\ge\ells>-\infty$ is met for all $\varphi\in\rew^\star$, as
in the following example.
\begin{example}
\label{sublinearity}
  The bound $z(\varphi)\ge\ells>-\infty$ holds for all
$\varphi\in\rew^\star$ if $\prob[S_1<\infty]=1$,
$\liminf_{s\uparrow\infty}v(s)/s=0$, and there exists a positive real
function $g$ on $\{1,2,\ldots\}\cup\{\infty\}$ such that
$\lim_{s\uparrow\infty}g(s)/s=0$ and $\|X_1\|\le g(S_1)$ with full
probability. Indeed, given any $\zeta<\ells$, under these hypotheses
one can find $\epsilon>0$ such that $\zeta+\epsilon<\ells\le 0$ and
$-\|\varphi\|g(s)+v(s)\ge-\epsilon s$ for all sufficiently large
$s$. Then, $\Ex[e^{\varphi(X_1)+v(S_1)-\zeta
    S_1}\mathds{1}_{\{S_1<\infty\}}]\ge\sum_{s\ge
  1}e^{-\|\varphi\|g(s)+v(s)-\zeta
  s}p(s)\ge\sum_{s>t}e^{-(\zeta+\epsilon)s}p(s)\ge
e^{-(\zeta+\epsilon)t}\prob[S_1>t]$ for all sufficiently large $t$ as
$\prob[S_1=\infty]=0$. It follows that
$\Ex[e^{\varphi(X_1)+v(S_1)-\zeta
    S_1}\mathds{1}_{\{S_1<\infty\}}]=+\infty$ since
$\zeta+\epsilon<\ells$, which results in $z(\varphi)\ge\ells$
according to definition (\ref{defzetafun}).
\end{example}

In order to establish large deviation bounds with respect to the PM,
it is convenient to distinguish the case $\ells=-\infty$ from the case
$\ells>-\infty$. The following theorem, which represents our second
main result, provides weak and full LDPs for the renewal-reward
process $t\mapsto W_t$ with respect to the PM when $\ells=-\infty$.
The proof is given in Section \ref{sec:freemodels}.
\begin{theorem}
\label{mainth2}
Assume $\ells=-\infty$. The following conclusions hold:
\begin{enumerate}[{\upshape(a)}]
\item if $G\subseteq\rew$ is open, then
\begin{equation*}
\liminf_{t\uparrow\infty}\frac{1}{t}\ln\prob_t\bigg[\frac{W_t}{t}\in G\bigg]\ge -\inf_{w\in G}\big\{I(w)\big\};
\end{equation*}
\item if $F\subseteq\rew$ is compact, then
\begin{equation*}
\limsup_{t\uparrow\infty}\frac{1}{t}\ln\prob_t\bigg[\frac{W_t}{t}\in F\bigg]\le -\inf_{w\in F}\big\{I(w)\big\}.
\end{equation*}
If $F\subseteq\rew$ is open convex, closed convex, or any convex set
in $\Brew$ when $\rew$ is finite-dimensional, then this bound is valid
whenever $I(0)<+\infty$. Furthermore, if $\rew$ is finite-dimensional,
then it is valid for any closed set $F$ provided that
$\Ex[e^{\xi\|X_1\|+v(S_1)-\zeta
    S_1}\mathds{1}_{\{S_1<\infty\}}]<+\infty$ for some numbers
$\zeta\ge 0$ and $\xi>0$.
\end{enumerate}
\end{theorem}
In general, the large deviation upper bound in part (b) cannot be
extended to convex sets if $\ells=-\infty$ and $I(0)=+\infty$. Examples
with an open convex set and a closed convex set where such bound fails
will be shown at the end of Section \ref{sec:freemodels}.

The case $\ells>-\infty$ is more involved and calls for two rate
functions, $\Ii$ and $\Is$, which are defined for each $w\in\rew$ by
the formulas
\begin{equation}
  \Ii(w):=\sup_{\varphi\in\rew^\star}\Big\{\varphi(w)-z(\varphi)\vee\elli+z(0)\vee\ells\Big\}
\label{def:Ibasso}
\end{equation}
and
\begin{equation}
  \Is(w):=\sup_{\varphi\in\rew^\star}\Big\{\varphi(w)-z(\varphi)\vee\ells+z(0)\vee\elli\Big\}.
\label{def:Ialto}
\end{equation}
The following theorem, which is our third and last main result,
provides large deviation bounds with respect to the PM when
$\ells>-\infty$.  The proof is reported in Section
\ref{sec:freemodels}.
\begin{theorem}
\label{mainth3}
Assume $\ells>-\infty$. The following conclusions hold:
\begin{enumerate}[{\upshape(a)}]
\item the functions $\Ii$ and $\Is$ are lower
  semicontinuous and proper convex;
\item if $G\subseteq\rew$ is open, then
\begin{equation*}
\liminf_{t\uparrow\infty}\frac{1}{t}\ln\prob_t\bigg[\frac{W_t}{t}\in G\bigg]\ge -\inf_{w\in G}\big\{\Ii(w)\big\};
\end{equation*}
\item if $F\subseteq\rew$ is compact, open convex, closed convex, or
  any convex set in $\Brew$ when $\rew$ is finite-dimensional, then
\begin{equation*}
\limsup_{t\uparrow\infty}\frac{1}{t}\ln\prob_t\bigg[\frac{W_t}{t}\in F\bigg]\le -\inf_{w\in F}\big\{\Is(w)\big\}.
\end{equation*}
Furthermore, if $\rew$ is finite-dimensional, then this bound is valid
for any closed set $F$ provided that $\Ex[e^{\xi\|X_1\|+v(S_1)-\zeta
    S_1}\mathds{1}_{\{S_1<\infty\}}]<+\infty$ for some numbers
$\zeta\ge 0$ and $\xi>0$.
\end{enumerate}
\end{theorem}
Theorem \ref{mainth3} states that the renewal-reward process $t\mapsto
W_t$ satisfies a weak LDP with rate function $\Is$ within the PM
provided that $\Ii=\Is$. The exponential moment condition
$\Ex[e^{\xi\|X_1\|+v(S_1)-\zeta
    S_1}\mathds{1}_{\{S_1<\infty\}}]<+\infty$ for some $\zeta\ge 0$
and $\xi>0$ gives a full LDP with rate function $\Is$ when $\rew$ is
finite-dimensional and $\Ii=\Is$. We have $\Ii=\Is$ if $\elli=\ells$,
as expected in most applications, or if the condition
$z(\varphi)\ge\ells$ is fulfilled for all $\varphi\in\rew^\star$, as
in example \ref{sublinearity}. In the latter case, $\Ii=\Is=I$.

\subsection{Discussion}

Large deviations for renewal-reward processes have been investigated
by many authors over the past decades. Their attention has been
focused on both discrete-time and continuous-time frameworks and, in
most cases, on rewards taking real values. In order to fix ideas, when
talking about renewal systems in the domain of time we think of a PM
with waiting times satisfying $\prob[S_1<\infty]=1$ and potential
$v=0$.  An almost omnipresent hypothesis in previous works is the
Cram\'er condition $\Ex[e^{\xi\|X_1\|+\xi S_1}]<+\infty$ for some
number $\xi>0$.

The simplest example of renewal-reward process has unit rewards and
corresponds to the counting renewal process $t\mapsto N_t:=\sum_{i\ge
  1}\mathds{1}_{\{T_i\le t\}}$. Glynn and Whitt \cite{Glynn1994}
investigated the connection between LDPs of the inverse processes
$t\mapsto N_t$ and $i\mapsto T_i$, providing a full LDP for $N_t$
under the Cram\'er condition. This condition was later relaxed by
Duffield and Whitt \cite{Duffield1998}. Jiang \cite{Jiang1994} studied
the large deviations of the extended counting renewal process
$t\mapsto\sum_{i\ge 1}\mathds{1}_{\{T_i\le i^\alpha t\}}$ with
$\alpha\in[0,1)$ under the Cram\'er condition.  Glynn and Whitt
  \cite{Glynn1994} and Duffield and Whitt \cite{Duffield1998},
  together with Puhalskii and Whitt \cite{Puhalskii1997}, also
  investigated the connection between sample-path LDPs of the
  processes $t\mapsto N_t$ and $i\mapsto T_i$ under the Cram\'er
  condition.

Starting from sample-path LDPs of inverse and compound processes,
Duffy and Rodgers-Lee \cite{Duffy2004} sketched a full LDP for
renewal-reward processes with real rewards by means of the contraction
principle under the stringent exponential moment condition
$\Ex[e^{\xi\|X_1\|+\xi S_1}]<+\infty$ for all $\xi>0$. Some full LDPs
for real renewal-reward processes were later proposed by Macci
\cite{Macci2005,Macci2007} under existence and essentially smoothness
of the scaled cumulant generating function, which allow for an
application of the G\"artner-Ellis theorem \cite{DemboBook,Hollander}.
Essentially smoothness of the scaled cumulant generating function has
been recently relaxed by Borovkov and Mogulskii
\cite{Borovkov2015,Borovkov2019}, which used the Cram\'er's theorem to
establish a full LDP under the Cram\'er condition. Under the Cram\'er
condition, they \cite{Borovkov2016_1,Borovkov2016_2,Borovkov2016_3}
have also obtained sample-path LDPs for real renewal-reward processes.

A different approach based on empirical measures has been considered
by Lefevere, Mariani, and Zambotti \cite{Lefevere2011}, which have
investigated large deviations for the empirical measures of forward
and backward recurrence times associated with a renewal process, and
have then derived by contraction a full LDP for renewal-rewards
processes with rewards determined by the waiting times: $X_i:=f(S_i)$
for each $i$ with a bounded real function $f$. Later, Mariani and
Zambotti \cite{Mariani2014} have developed a renewal version of
Sanov's theorem by studying the empirical law of rewards that take
values in a generic Polish space. By appealing to the contraction
principle, this result could give a full LDP for a renewal-reward
process with rewards valued in a real separable Banach space, but only
provided that the strong exponential moment condition
$\Ex[e^{\xi\|X_1\|}]<+\infty$ for all $\xi>0$ is satisfied as
discussed by Schied \cite{Schied1998}.

We conclude this brief review of previous contributions by mentioning
that a moderate deviation principle for real renewal-reward processes
was obtained by Tsirelson \cite{Tsirelson2013} under an exponential
moment condition.  Exact asymptotics for the counting renewal process
and real renewal-reward processes has been investigated under the
Cram\'er condition and several additional smoothness hypotheses by
Serfozo \cite{Serfozo1974}, Kuczek and Crank \cite{Kuczek1991}, Chi
\cite{Chi2007}, and Borovkov and Mogulskii
\cite{Borovkov2018_1,Borovkov2018_2}.

Previous works leave open the question of whether some large deviation
principles free from exponential moment conditions can be established
for renewal-reward processes, in the wake of the sharp version of
Cram\'er's theorem demonstrated by Bahadur and Zabell \cite{BaZa}. The
present paper gives a positive answer to this question at the price of
restricting to the discrete-time framework. Indeed, through theorems
\ref{mainth1}, \ref{mainth2}, and \ref{mainth3} we supply weak LDPs
and large deviation upper bounds for measurable convex sets that are
completely free from hypotheses. Moreover, when finite-dimensional
rewards are considered, and when $\prob[S_1<\infty]=1$ and $v=0$ to
make a comparison with previous studies, we provide full LDPs under
the exponential moment condition $\Ex[e^{\xi\|X_1\|-\zeta
    S_1}]<+\infty$ for some numbers $\zeta\ge 0$ and $\xi>0$, which is
weaker than the Cram\'er condition $\Ex[e^{\xi\|X_1\|+\xi
    S_1}]<+\infty$ for some $\xi>0$.  For instance, rewards of example
\ref{zfintutto} that define the macroscopic observables of Statistical
Mechanics \cite{Models} always satisfy our weak exponential moment
condition, whereas in general they do not fulfill the Cram\'er
condition.

In order to drop exponential moment conditions, a novel approach with
respect to past methods had to be devised to tackle the problem, and a
new approach was suggested to us by the theory of polymer pinning
\cite{Giac,Hollander_Polymer}. This new approach is based on
super-additivity, but requires discrete time to be implemented. It
came from here the need to focus on the discrete-time framework.  In
such framework, conditioning on the event that the last time is a
renewal time is a meaningful procedure and enables a super-additivity
property of renewal-reward processes to emerge.  This procedure
introduces a constrained model similarly to what is done with
polymers.  This way, we were able to find a successful strategy for
investigating large deviations and we were naturally led to link
renewal-reward processes to the PM and the CPM. Importantly, the CPM
is not a merely mathematical tool to tackle the PM, but it also
represents the renewal models of Statistical Mechanics \cite{Models},
such as the Poland-Scheraga model, the Fisher-Felderhof model, the
Wako-Sait\^o-Mu\~noz-Eaton model, and the Tokar-Dreyss\'e model.  In
this respect, the large deviation theory developed in this paper must
be added to those already existing for other models of Statistical
Mechanics, including the Curie-Weiss model \cite{Ellisbook}, the
Curie-Weiss-Potts model \cite{CoElTou}, the mean-field
Blume-Emery-Griffiths model \cite{ElOtTou}, and, to some extent, the
Ising model as well as general Gibbs measures relative to an
interaction potential \cite{El,FoOr,Ol,Geo}.

Going back for a moment to the domain of time with
$\prob[S_1<\infty]=1$ and $v=0$, it is interesting to point out that
Duffy and Rodgers-Lee \cite{Duffy2004}, Lefevere, Mariani, and
Zambotti \cite{Lefevere2011}, and Borovkov and Mogulskii
\cite{Borovkov2015,Borovkov2019} found, with increasing level of
generality, an apparently different rate function. They constructed
the rate function for renewal-reward processes from the Cram\'er rate
function $\LC$ of the waiting time and reward pair $(S_1,X_1)$,
defined for each pair $(\beta,w)\in\Rl\times\rew$ by
\begin{equation*}
\LC(\beta,w):=\sup_{(\zeta,\varphi)\in\Rl\times\rew^\star}\Big\{\varphi(w)-\beta\zeta-\ln\Ex\big[e^{\varphi(X_1)-\zeta S_1}\big]\Big\}.
\end{equation*}
Starting from $\LC$, they considered the function
$\inf_{\gamma>0}\{\gamma\LC(\cdot/\gamma,\cdot/\gamma)\}$, whose
lower-semicontinuous regularization $\Upsilon$ is given for every
$(\beta,w)\in\Rl\times\rew$ by
\begin{equation*}
  \Upsilon(\beta,w):=\adjustlimits\lim_{\delta\downarrow 0}\inf_{\alpha\in (\beta-\delta,\beta+\delta)}\inf_{u\in B_{w,\delta}}\inf_{\gamma>0}
  \big\{\gamma\LC(\alpha/\gamma,u/\gamma)\big\}.
\end{equation*}
Here $B_{w,\delta}:=\{u\in\rew:\|u-w\|<\delta\}$ is the open ball of
center $w$ and radius $\delta$.  Duffy and Rodgers-Lee
\cite{Duffy2004} dealt with the case $\rew=\Rl$ and $\ells=-\infty$
under a strong exponential moment condition, obtaining the rate
function $\Lambda:=\Upsilon(1,\cdot\,)$.  In this case we have the
rate function $I$ by theorem \ref{mainth2} with
$z(\varphi)=\inf\{\zeta\in\Rl:\Ex[e^{\varphi(X_1)-\zeta S_1}]\le 1\}$
for all $\varphi$ as $S_1<\infty$ with full probability and $v=0$.
Lefevere, Mariani, and Zambotti \cite{Lefevere2011} too found the rate
function $\Lambda$ when $X_1=f(S_1)$ with a bounded real function
$f$. This instance falls under the umbrella of example
\ref{sublinearity}, and so we get again the rate function $I$ by
theorem \ref{mainth3}. Borovkov and Mogulskii
\cite{Borovkov2015,Borovkov2019} studied the case $\rew=\Rl$ under the
Cram\'er condition. They obtained the rate function $\Lambda$ when
$\ells=-\infty$ and the rate function
$\Lambdas:=\inf_{\beta\in[0,1]}\{\Upsilon(\beta,\cdot)-(1-\beta)\ells\}$
when $\elli=\ells>-\infty$ or $z(\varphi)\ge\ells>-\infty$ for all
$\varphi\in\rew^\star$. In these cases we have the rate functions $I$
and $\Is$, respectively, by theorem \ref{mainth2} and \ref{mainth3}.
Despite different expressions, our results are consistent with all the
findings of these authors.  Indeed, while uniqueness of the rate
function \cite{DemboBook,Hollander} suggests that there is at least
some situation where $I=\Lambda$ and $\Is=\Lambdas$, a direct
comparison shows that these identities hold in general, as established
by the following lemma which is proved in \ref{proof:confrontoIL}.
\begin{lemma}
  \label{confrontoIL}
  Assume that $\prob[S_1<\infty]=1$ and that $v=0$. Then, the
  following conclusions hold for every $w\in\rew$:
  \begin{enumerate}[{\upshape(a)}]
  \item $I(w)=\Lambda(w):=\Upsilon(1,w)$;
  \item $\Is(w)=\Lambdas(w):=\inf_{\beta\in[0,1]}\{\Upsilon(\beta,w)-(1-\beta)\ells\}$ provided that $\ells>-\infty$.
    \end{enumerate}
\end{lemma}

As a final remark, we stress that Borovkov and Mogulskii
\cite{Borovkov2015,Borovkov2019} opted for not introducing two
different rate functions for the large deviation lower and upper
bounds, thus considering only problems where $\Ii=\Is$. At variance
with them, we decided to provide optimal large deviation bounds with
possibly different rate functions in order to even address situations
where the tail of the waiting time distribution is very
oscillating. For instance, a physical renewal model giving rise to two
possibly different rate functions has been found by Lefevere, Mariani,
and Zambotti \cite{Lefevere2011_2,Lefevere2012} in the description of
a free particle interacting with a heat bath.

\section{Proof of proposition \ref{proposition:free_energy} and theorem \ref{mainth1}}
\label{sec:constrainedmodels}

We prove proposition \ref{proposition:free_energy} and theorem
\ref{mainth1} as follows. In Section \ref{par:existence} we show the
existence of a weak LDP with a convex rate function.  This is the step
where convexity and super-additivity arguments come into play. In
Section \ref{par:renseq} we introduce the generalized renewal equation
formalism, which allows us to express the scaled cumulative generating
function in terms of the function $z$ defined by
(\ref{defzetafun}). Then, we use this formalism in Section
\ref{par:cI} to also relate the rate function to $z$. Finally, in
Section \ref{par:th1pp} we summarize the results linking them to
proposition \ref{proposition:free_energy} and to parts (a), (b), and
(c) of theorem \ref{mainth1}.

Our theory of large deviations take advantage of the fact that a
renewal process forgets the past and starts over at every renewal.
Concretely, this means that
$(U_{\tau+t},\Delta_\tau^tH,\Delta_\tau^tW)_{t\ge 1}$ with
$\Delta_\tau^tH:=H_{\tau+t}-H_\tau$ and
$\Delta_\tau^tW:=W_{\tau+t}-W_\tau$ is independent of
$(H_\tau,W_\tau)$ and distributed as $(U_t,H_t,W_t)_{t\ge 1}$
conditional on the event that a given integer $\tau\ge 1$ is a
renewal, namely
\begin{eqnarray}
\nonumber
\Ex\Big[\mathds{1}_{\{(H_\tau,W_\tau)\in\cdot\}}\mathds{1}_{\{(U_{\tau+t},\Delta_\tau^tH,\Delta_\tau^tW)_{t\ge 1}\in\star\}}U_\tau\Big]
&=&\Ex\Big[\mathds{1}_{\{(H_\tau,W_\tau)\in\cdot\}}U_\tau\Big]\\
  &\cdot&\Ex\Big[\mathds{1}_{\{(U_t,H_t,W_t)_{t\ge 1}\in\star\}}\Big].
\label{start_0}
\end{eqnarray}
A formal proof of (\ref{start_0}) can be drawn by noticing that if
$\tau=T_n$ for some positive integer $n$, then $T_i\le\tau$ for each
$i\le n$ and $T_i>\tau$ for any $i>n$. It follows that
$H_\tau=\sum_{i=1}^nv(S_i)$ and $W_\tau=\sum_{i=1}^nX_i$, so that the
random vector $(H_\tau,W_\tau)$ depends only on
$(S_1,X_1),\ldots,(S_n,X_n)$. At the same time, for any $t\ge 1$ we
have $U_{\tau+t}=\sum_{i\ge n+1}\mathds{1}_{\{T_i=\tau+t\}}=\sum_{i\ge
  1}\mathds{1}_{\{S_{n+1}+\cdots+S_{n+i}=t\}}$,
$\Delta_\tau^tH=\sum_{i\ge n+1}v(S_i)\mathds{1}_{\{T_i\le
  \tau+t\}}=\sum_{i\ge 1}
v(S_{n+i})\mathds{1}_{\{S_{n+1}+\cdots+S_{n+i}\le t\}}$, and
analogously $\Delta_\tau^tW=\sum_{i\ge n+1}X_i\mathds{1}_{\{T_i\le
  \tau+t\}}=\sum_{i\ge 1}
X_{n+i}\mathds{1}_{\{S_{n+1}+\cdots+S_{n+i}\le t\}}$, showing that the
vector $(U_{\tau+t},\Delta_\tau^tH,\Delta_\tau^tW)$ depends only on
$(S_{n+1},X_{n+1}),(S_{n+2},X_{n+2}),\ldots$ through the same formula
that connects $(U_t,H_t,W_t)$ to $(S_1,X_1),(S_2,X_2),\ldots$.

\subsection{Weak LDP in the constrained setting}
\label{par:existence}

We leave the normalizing constant $Z_t^c$ aside for the moment and
focus on the measure $\mu_t$ over $\Brew$ defined for each time $t\ge
1$ by
\begin{equation*}
\mu_t:=\Ex\bigg[\mathds{1}_{\big\{\frac{W_t}{t}\in \cdot\big\}} U_t e^{H_t}\bigg].
\end{equation*}
We have $\mu_t(\rew)=\Ex[U_t e^{H_t}]=Z_t^c>0$ for all $t>t_c$ and
some $t_c\ge 0$ thanks to assumption \ref{ass1} about aperiodicity, as
we have seen at the end of Section \ref{sec:basics}.
Of fundamental importance is the following super-multiplicativity
property, which is not fulfilled by
$\prob_t^c[W_t/t\in\cdot\,\,]=\mu_t/Z_t^c$ precisely because of
normalization.
\begin{lemma}
  Let $C\in\Brew$ be convex and let $\tau\ge 1$ and $t\ge 1$ be two
  integers. Then, $\mu_{\tau+t}(C)\ge\mu_\tau(C)\cdot\mu_t(C)$.
\end{lemma}

\begin{proof}
Writing $W_{\tau+t}/(\tau+t)=\lambda
W_\tau/\tau+(1-\lambda)\Delta_\tau^tW/t$ with
$\lambda:=\tau/(\tau+t)$, we recognize that $W_{\tau+t}/(\tau+t)\in C$
whenever $W_\tau/\tau\in C$ and $\Delta_\tau^tW/t\in C$ since $C$ is
convex. It follows that
\begin{eqnarray}
  \nonumber
  \mu_{\tau+t}(C)&=&\Ex\bigg[\mathds{1}_{\big\{\frac{W_{\tau+t}}{\tau+t}\in
    C \big\}}U_{\tau+t} e^{H_{\tau+t}}\bigg]\\
\nonumber
  &\ge&\Ex\bigg[\mathds{1}_{\big\{\frac{W_\tau}{\tau}\in
    C \big\}}\mathds{1}_{\big\{\frac{\Delta_\tau^tW}{t}\in C\big\}}
  U_{\tau+t} e^{H_{\tau+t}}\bigg]\\
\nonumber
&=&\Ex\bigg[\mathds{1}_{\big\{\frac{W_\tau}{\tau}\in
    C \big\}}e^{H_\tau}\mathds{1}_{\big\{\frac{\Delta_\tau^tW}{t}\in C\big\}}
  U_{\tau+t} e^{\Delta_\tau^tH}\bigg].
\end{eqnarray}
A looser lower bound is obtained by introducing the renewal indicator
$U_\tau$ with the motivation that
$(U_{\tau+t},\Delta_\tau^tH,\Delta_\tau^tW)$ is independent of
$(U_\tau,H_\tau,W_\tau)$ and distributed as $(U_t,H_t,W_t)$ when
$\tau$ is a renewal. This way, invoking (\ref{start_0}) we find
\begin{eqnarray}
  \nonumber
  \mu_{\tau+t}(C)&\ge&\Ex\bigg[\mathds{1}_{\big\{\frac{W_\tau}{\tau}\in
    C \big\}}U_\tau e^{H_\tau}\mathds{1}_{\big\{\frac{\Delta_\tau^tW}{t}\in C\big\}}
    U_{\tau+t} e^{\Delta_\tau^tH}\bigg]\\
\nonumber
  &=&\Ex\bigg[\mathds{1}_{\big\{\frac{W_\tau}{\tau}\in
    C \big\}}U_\tau e^{H_\tau}\bigg]\cdot\Ex\bigg[\mathds{1}_{\big\{\frac{W_t}{t}\in C\big\}}
  U_t e^{H_t}\bigg]\\
\nonumber
&=&\mu_\tau(C)\cdot\mu_t(C),
\end{eqnarray}
which proves the lemma.
\end{proof}

Super-multiplicativity, which becomes super-additivity once logarithms
are taken, makes it possible to describe in general terms the
exponential decay with $t$ of the measure $\mu_t$. To this purpose, we
denote by $\mathcal{L}$ the extended real function over $\Brew$
defined by the formula
\begin{equation*}
\mathcal{L}:=\sup_{t>t_c}\bigg\{\frac{1}{t}\ln \mu_t\bigg\}.
\end{equation*}
If $C\in\Brew$ is convex, then the super-additivity of $\ln \mu_t(C)$
immediately gives $\limsup_{t\uparrow\infty}(1/t)\ln
\mu_t(C)=\mathcal{L}(C)$. The following lemma improves this result
when $C$ is open as well as convex. Hereafter we denote by
$B_{w,\delta}:=\{u\in\rew:\|u-w\|<\delta\}$ the open ball of center
$w$ and radius $\delta$, which is an example of open convex set.
\begin{lemma}
  \label{start}
  Let $C\subseteq\rew$ be open and convex. Then,
  $\lim_{t\uparrow\infty}(1/t)\ln \mu_t(C)$ exists as an extended real
  number and is equal to $\mathcal{L}(C)$.
\end{lemma}

\begin{proof}
We shall show in a moment that the hypothesis that $C$ is open entails
that either $\mu_t(C)=0$ for all $t>t_c$ or there exists $\tau\ge t_c$
such that $\mu_t(C)>0$ for all $t>\tau$. Lemma \ref{start} is obvious
in the first case. The second case is solved as follows. Pick an
integer $s>t_c$. Then, fix an integer $\gamma\ge 1$ such that $\gamma
s>\tau$ and a constant $M>-\infty$ such that $\ln\mu_r(C)\ge M$ when
$\gamma s\le r<2\gamma s$, which exists because $\gamma
s>\tau$. Expressing any $t\ge 2\gamma s$ as $t=q\gamma s+r$ with $q\ge
1$ and $\gamma s\le r<2\gamma s$, super-additivity gives
$\ln\mu_t(C)\ge q\gamma\ln\mu_s(C)+\ln\mu_r(C)\ge
q\gamma\ln\mu_s(C)+M$, thus showing that
$\liminf_{t\uparrow\infty}(1/t)\ln\mu_t(C)\ge(1/s)\ln\mu_s(C)$.  The
arbitrariness of $s$ yields
$\liminf_{t\uparrow\infty}(1/t)\ln\mu_t(C)\ge\sup_{s>t_c}\{(1/s)\ln\mu_s(C)\}=:\mathcal{L}(C)$.

We now prove that either $\mu_t(C)=0$ for all $t>t_c$ or there exists
$\tau\ge t_c$ with the property that $\mu_t(C)>0$ for all
$t>\tau$. Assume that $\mu_{\tau_o}(C)>0$ for some $\tau_o>t_c$.  To
begin with, we notice that if for every $w\in C$ it were possible to
find a number $\delta_w>0$ such that $\mu_{\tau_o}(B_{w,\delta_w})=0$,
then the open covering $\{B_{w,\delta_w}\}_{w\in C}$ of $C$ would
contain a countable subcollection covering $C$ by separability of
$\rew$ and Lindel\"of's lemma with the consequence that
$\mu_{\tau_o}(C)=0$.  This argument shows that there exists at least
one point $w_o\in C$ such that $\mu_{\tau_o}(B_{w_o,\delta})>0$ for
all $\delta>0$. Since $C$ is open, there is $\delta_o>0$ such that
$B_{w_o,2\delta_o}\subseteq C$.  This way, we have constructed open
balls $B_k:=B_{w_o,k\delta_o}$ so that $\mu_{\tau_o}(B_1)>0$ and
$B_2\subseteq C$. Furthermore, since
$\lim_{k\uparrow\infty}\mu_r(B_k)=\mu_r(\rew)=Z_r^c>0$ for all
$r>t_c$, there exists an integer $k_o\ge 1$ such that
$\mu_r(B_{k_o})>0$ if $r$ satisfies $\tau_o\le r<2\tau_o$. Set
$\tau:=2k_o\tau_o$.

Let us pick an arbitrary $t>\tau$ and let us show that $\mu_t(C)>0$.
The fact that $t>\tau\ge 2\tau_o$ makes it possible to express $t$ as
$t=q\tau_o+r$ with integers $q$ and $r$ such that $q\ge 1$ and
$\tau_o\le r<2\tau_o$. We notice that $W_t/t\in B_2$ whenever
$W_{q\tau_o}/q\tau_o\in B_1$ and $\Delta_{q\tau_o}^rW/r\in B_{k_o}$,
as the following bounds demonstrate:
\begin{eqnarray}
\nonumber
\big\|W_t-tw_o\big\|&\le&\big\|W_{q\tau_o}-q\tau_ow_o\big\|+\big\|\Delta_{q\tau_o}^rW-rw_o\big\|\\
\nonumber
&<&\delta_o(q\tau_o+k_or)<\delta_o(t+2k_o\tau_o)=\delta_o(t+\tau)<2\delta_ot.
\end{eqnarray}
Then, recalling that $B_2\subseteq C$ we get
\begin{eqnarray}
  \nonumber
  \mu_t(C)&\ge&\Ex\bigg[\mathds{1}_{\big\{\frac{W_t}{t}\in B_2 \big\}}U_t e^{H_t}\bigg]\\
\nonumber
  &\ge&\Ex\bigg[\mathds{1}_{\big\{\frac{W_{q\tau_o}}{q\tau_o}\in
    B_1 \big\}}\mathds{1}_{\big\{\frac{\Delta_{q\tau_o}^rW}{r}\in B_{k_o}\big\}} U_t e^{H_t}\bigg]\\
\nonumber
&=&\Ex\bigg[\mathds{1}_{\big\{\frac{W_{q\tau_o}}{q\tau_o}\in
    B_1 \big\}}e^{H_{q\tau_o}}\mathds{1}_{\big\{\frac{\Delta_{q\tau_o}^rW}{r}\in B_{k_o}\big\}} U_{q\tau_o+r}\,e^{\Delta_{q\tau_o}^rH}\bigg].
\end{eqnarray}
As in the proof of lemma \ref{start}, a convenient looser lower bound
is obtained by introducing $U_{q\tau_o}$. Since
$(U_{q\tau_o+r},\Delta_{q\tau_o}^rH,\Delta_{q\tau_o}^rW)$ is independent of
$(U_{q\tau_o},H_{q\tau_o},W_{q\tau_o})$ and distributed as
$(U_r,H_r,W_r)$ when $q\tau_o$ is a renewal we find
\begin{eqnarray}
  \nonumber
  \mu_t(C)&\ge&\Ex\bigg[\mathds{1}_{\big\{\frac{W_{q\tau_o}}{q\tau_o}\in
    B_1 \big\}}U_{q\tau_o}e^{H_{q\tau_o}}\mathds{1}_{\big\{\frac{\Delta_{q\tau_o}^rW}{r}\in B_{k_o}\big\}} U_{q\tau_o+r}\,e^{\Delta_{q\tau_o}^rH}\bigg]\\
\nonumber
 &=&\Ex\bigg[\mathds{1}_{\big\{\frac{W_{q\tau_o}}{q\tau_o}\in
    B_1 \big\}}U_{q\tau_o}e^{H_{q\tau_o}}\bigg]\cdot\Ex\bigg[\mathds{1}_{\big\{\frac{W_r}{r}\in B_{k_o}\big\}} U_r e^{H_r}\bigg]\\
\nonumber
&=&\mu_{q\tau_o}(B_1)\cdot\mu_r(B_{k_o})\ge\mu^q_{\tau_o}(B_1)\cdot\mu_r(B_{k_o}),
\end{eqnarray}
where the last inequality is due to super-multiplicativity because
$B_1$ is convex. We deduce from here that $\mu_t(C)>0$ as both
$\mu_{\tau_o}(B_1)>0$ and $\mu_r(B_{k_o})>0$ by construction.
\end{proof}

Lemma \ref{start} suggests to consider the putative rate function $J$
that maps any $w\in\rew$ in the extended real number $J(w)$ defined
by
\begin{equation*}
J(w):=-\inf_{\delta>0}\big\{\mathcal{L}(B_{w,\delta})\big\}.
\end{equation*}
In fact, the function $J$ controls the measure decay of open and
compact sets as follows.
\begin{proposition}
\label{WLDP}
The following conclusions hold:
\begin{enumerate}[{\upshape(i)}]
\item $\displaystyle{\liminf_{t\uparrow\infty}\frac{1}{t}\ln\mu_t(G)\ge -\inf_{w\in G}\{J(w)\}}$ for each $G\subseteq\rew$ open;
\item $\displaystyle{\limsup_{t\uparrow\infty}\frac{1}{t}\ln\mu_t(K)\le -\inf_{w\in K}\{J(w)\}}$ for each $K\subseteq\rew$ compact.
\end{enumerate} 
\end{proposition}

\begin{proof}
Part (i) is immediate.  Let $G\subseteq\rew$ be open, let $w\in G$ be
an arbitrary point, and let $\delta>0$ be such that
$B_{w,\delta}\subseteq G$.  Since $\mu_t(G)\ge\mu_t(B_{w,\delta})$ and
since $B_{w,\delta}$ is open and convex, lemma \ref{start} gives
$\liminf_{t\uparrow\infty}(1/t)\ln
\mu_t(G)\ge\lim_{t\uparrow\infty}(1/t)\ln
\mu_t(B_{w,\delta})=\mathcal{L}(B_{w,\delta})\ge-J(w)$. The conclusion
follows from the arbitrariness of $w$.

Moving to part (ii), pick a compact set $K$ in $\rew$ and assume
$\inf_{w\in K}\{J(w)\}>-\infty$, otherwise there is nothing to
prove. Let $\lambda<\inf_{w\in K}\{J(w)\}$ be a real number.  Since
there exists $\epsilon>0$ such that $\lambda+\epsilon\le
J(w)=-\inf_{\delta>0}\{\mathcal{L}(B_{w,\delta})\}$ for every $w\in
K$, a number $\delta_w>0$ can be found for each $w\in K$ in such a way
that $\mathcal{L}(B_{w,\delta_w})\le -\lambda$. Then, lemma
\ref{start} yields
$\lim_{t\uparrow\infty}(1/t)\ln\mu_t(B_{w,\delta_w})\le-\lambda$ for
such $\delta_w$. Due to the compactness of $K$, there exist finitely
many points $w_1,\ldots,w_n$ in $K$ such that $K\subseteq\cup_{i=1}^n
B_{w_i,\delta_{w_i}}$. It follows that
$\mu_t(K)\le\sum_{i=1}^n\mu_t(B_{w_i,\delta_{w_i}})$, which in turn
gives $\limsup_{t\uparrow\infty}(1/t)\ln\mu_t(K)\le-\lambda$.  This
way, we get the desired upper bound by sending $\lambda$ to
$\inf_{w\in K}\{J(w)\}$.
\end{proof}

The first important properties of $J$ are presented in the following
lemma.
\begin{lemma}
\label{Jprop}
The function $J$ is lower semicontinuous and convex.
\end{lemma}

\begin{proof}
Pick $w\in\rew$ and let $\{w_i\}_{i\ge 0}$ be a sequence of points
converging to $w$. We show that
$\liminf_{i\uparrow\infty}J(w_i)\ge-\mathcal{L}(B_{w,\delta})$ for all
numbers $\delta>0$, which results in
$\liminf_{i\uparrow\infty}J(w_i)\ge J(w)$ and proves the lower
semicontinuity of $J$. Given $\delta>0$ there exists $i_o\ge 0$ such
that $\|w_i-w\|\le\delta/2$ if $i\ge i_o$. Then, monotonicity of
$\mathcal{L}$ inherited from the measures $\mu_t$ entails that
$-J(w_i)\le\mathcal{L}(B_{w_i,\delta/2})\le\mathcal{L}(B_{w,\delta})$
for each $i\ge i_o$ since $B_{w_i,\delta/2}\subseteq B_{w,\delta}$.
The bound
$\liminf_{i\uparrow\infty}J(w_i)\ge-\mathcal{L}(B_{w,\delta})$ follows
from here.

As far as the proof of the convexity of $J$ is concerned, lower
semicontinuity combined with the fact that dyadic rationals in $[0,1]$
are dense in $[0,1]$ makes it sufficient to verify that for each $u$
and $w$ in $\rew$
\begin{equation}
  J\bigg(\frac{u+w}{2}\bigg)\le \frac{J(u)+J(w)}{2}.
\label{Jprop_1}
\end{equation}
To this aim, we notice that for each number $\delta>0$ and integer
$t\ge 1$ the conditions $W_t/t\in B_{u,\delta}$ and $\Delta_t^tW/t\in
B_{w,\delta}$ imply $W_{2t}/(2t)\in B_{(u+w)/2,\delta}$, as one can
easily verify.  It follows that
\begin{eqnarray}
\nonumber
\mu_{2t}\Big(B_{\frac{u+w}{2},\delta}\Big)&=&\Ex\bigg[\mathds{1}_{\big\{\frac{W_{2t}}{2t}\in
    B_{\frac{u+w}{2},\delta}\big\}}U_{2t} e^{H_{2t}}\bigg]\\
\nonumber
&\ge&\Ex\bigg[\mathds{1}_{\big\{\frac{W_t}{t}\in B_{u,\delta}\big\}}\mathds{1}_{\big\{\frac{\Delta_t^tW}{t}\in B_{w,\delta}\big\}}U_{2t} e^{H_{2t}}\bigg]\\
\nonumber
&=&\Ex\bigg[\mathds{1}_{\big\{\frac{W_t}{t}\in B_{u,\delta}\big\}} e^{H_t}\mathds{1}_{\big\{\frac{\Delta_t^tW}{t}\in B_{w,\delta}\big\}}U_{2t} e^{\Delta_t^tH}\bigg].
\end{eqnarray}
Inserting $U_t$ and exploiting the fact that
$(U_{2t},\Delta_t^tH,\Delta_t^tW)$ is independent of $(U_t,H_t,W_t)$
and distributed as $(U_t,H_t,W_t)$ when $t$ is a renewal we get
\begin{eqnarray}
  \nonumber
  \mu_{2t}\Big(B_{\frac{u+w}{2},\delta}\Big)&\ge&\Ex\bigg[\mathds{1}_{\big\{\frac{W_t}{t}\in B_{u,\delta}\big\}}U_t e^{H_t}
    \mathds{1}_{\big\{\frac{\Delta_t^tW}{t}\in B_{w,\delta}\big\}}U_{2t} e^{\Delta_t^tH}\bigg]\\
  \nonumber
  &=&\Ex\bigg[\mathds{1}_{\big\{\frac{W_t}{t}\in B_{u,\delta}\big\}}U_t e^{H_t}\bigg]\cdot
  \Ex\bigg[\mathds{1}_{\big\{\frac{W_t}{t}\in B_{w,\delta}\big\}}U_t e^{H_t}\bigg]\\
  \nonumber
  &=&\mu_t(B_{u,\delta})\cdot\mu_t(B_{w,\delta}).
\end{eqnarray}
This way, taking logarithms, dividing by $2t$, and sending $t$ to
infinity, we find
$\mathcal{L}(B_{(u+w)/2,\delta})\ge(1/2)\mathcal{L}(B_{u,\delta})+(1/2)\mathcal{L}(B_{w,\delta})\ge
-(1/2)J(u)-(1/2)J(w)$ thanks to lemma \ref{start} because open balls
are open convex sets. Inequality (\ref{Jprop_1}) follows from here by
the arbitrariness of $\delta$.
\end{proof}

We conclude the section strengthening proposition \ref{WLDP} for
convex sets. We know that
$\limsup_{t\uparrow\infty}(1/t)\ln\mu_t(C)=\mathcal{L}(C)$ for every
$C\in\Brew$ convex thanks to super-additivity. The following lemma
draws a link between $\mathcal{L}(C)$ and $\inf_{w\in C}\{J(w)\}$.
\begin{lemma}
\label{lemma_conv}
Let $C\subseteq\rew$ be open convex, closed convex, or any convex set
in $\Brew$ when $\rew$ is finite-dimensional.  Then,
$\mathcal{L}(C)\le-\inf_{w\in C}\{J(w)\}$.
\end{lemma}
\begin{proof}
The lemma is trivial if $\mathcal{L}(C)=-\infty$. Assume
$\mathcal{L}(C)>-\infty$ and pick $\epsilon>0$. Since
$\limsup_{t\uparrow\infty}(1/t)\ln\mu_t(C)=\mathcal{L}(C)$ there
exists an integer $\tau\ge 1$ such that
$\mathcal{L}(C)\le(1/\tau)\ln\mu_\tau(C)+\epsilon$.  Completeness and
separability of $\rew$ entail that $\mu_\tau$ is tight as it is
bounded from above by $Z_\tau^c<+\infty$ (see \cite{Bogachevbook},
theorem 7.1.7).  Consequently, a compact set $K_o\subseteq C$ can be
found so that $\mu_\tau(C)\le
\mu_\tau(K_o)+[1-\exp(-\epsilon\tau)]\mu_\tau(C)$.  Thus,
$\mu_\tau(C)\le\exp(\epsilon\tau)\mu_\tau(K_o)$ and
$\mathcal{L}(C)\le(1/\tau)\ln\mu_\tau(K_o)+2\epsilon$ follows.  We
shall show in a moment that there exists a compact convex set $K$ with
the property that $K_o\subseteq K\subseteq C$.  Then, using the fact
that $K_o\subseteq K$ we reach the further bound
$\mathcal{L}(C)\le(1/\tau)\ln\mu_\tau(K)+2\epsilon\le\mathcal{L}(K)+2\epsilon$.
At this point, we notice that on the one hand
$\mathcal{L}(K)=\limsup_{t\uparrow\infty}(1/t)\ln\mu_t(K)$ by
super-additivity as $K$ is convex, and on the other hand
$\limsup_{t\uparrow\infty}(1/t)\ln\mu_t(K)\le-\inf_{w\in K}\{J(w)\}$
by part (ii) of proposition \ref{WLDP} as $K$ is compact. Thus,
$\mathcal{L}(C)\le-\inf_{w\in K}\{J(w)\}+2\epsilon\le-\inf_{w\in
  C}\{J(w)\}+2\epsilon$ because $K\subseteq C$ and the lemma follows
from the arbitrariness of $\epsilon$.

Let us prove now that there exists a compact convex set $K$ with the
property that $K_o\subseteq K\subseteq C$.  The hypothesis that the
convex set $C$ is either open or closed when $\rew$ if
infinite-dimensional comes into play here.  Let $C_o$ be the convex
hull of $K_o$ and let $K:=\cl C_o$, $\cl A$ denoting the closure of a
set $A$. Clearly, $K_o\subseteq C_o\subseteq C$ and $C_o\subseteq
K$. Since $K_o$ is compact, $C_o$ is convex and compact whenever
$\rew$ is finite-dimensional, whereas $K$ is convex and compact in any
circumstance (see \cite{Rudin}, theorem 3.20). We want to demonstrate
that $K\subseteq C$.  If $\rew$ is finite-dimensional, then $K=C_o$
and we get the desired result from $C_o\subseteq C$. If $\rew$ is
infinite-dimensional and $C$ is closed, then $K\subseteq C$ follows
from $C_o\subseteq C$ by taking closures. The only nontrivial case is
when $\rew$ is infinite-dimensional and $C$ is open. Assume that $C$
is open from now on and for each $w\in C$ let $\delta_w>0$ be such
that $B_{w,2\delta_w}\subseteq C$. As $K_o$ is compact, there exist
finitely many points $w_1,\ldots,w_n$ in $K_o$ so that
$K_o\subseteq\cup_{i=1}^nB_{w_i,\delta_{w_i}}$. Let $K'$ be the convex
hull of $\cup_{i=1}^n(\cl B_{w_i,\delta_{w_i}}\cap K)$.  We have
$K'\subseteq C$ because $\cup_{i=1}^n(\cl B_{w_i,\delta_{w_i}}\cap
K)\subseteq\cup_{i=1}^n\cl B_{w_i,\delta_{w_i}}\subseteq C$ thanks to
the fact that $B_{w_i,2\delta_{w_i}}\subseteq C$ for every $i$ and
because $C$ is convex. This way, $K\subseteq C$ is verified if we show
that $K=K'$. The inclusion $K'\subseteq K$ is immediate since
$\cup_{i=1}^n(\cl B_{w_i,\delta_{w_i}}\cap K)\subseteq K$ and $K$ is
convex. In order to show the opposite inclusion $K\subseteq K'$ we
observe that the set $K'$ is convex and compact since it is the convex
hull of the union of the compact convex sets $\cl
B_{w_1,\delta_{w_1}}\cap K,\ldots,\cl B_{w_n,\delta_{w_n}}\cap K$ (see
\cite{Rudin}, theorem 3.20).  Then, we observe that $K_o\subseteq K'$
as $\cup_{i=1}^n(\cl B_{w_i,\delta_{w_i}}\cap K)=(\cup_{i=1}^n\cl
B_{w_i,\delta_{w_i}})\cap K$ and both $\cup_{i=1}^n\cl
B_{w_i,\delta_{w_i}}$ and $K$ contain $K_o$.  This way, we first
realize that $C_o\subseteq K'$ since $C_o$ is the smallest convex set
that contains $K_o$, and by taking closures we later deduce that
$K\subseteq K'$ as $K'$ is closed.
\end{proof}

\subsection{Expectations and generalized renewal equation}
\label{par:renseq}

Let $(S_1,V_1),(S_2,V_2),\ldots$ be a sequence of independent and
identically distributed random vectors on
$(\Omega,\mathcal{F},\prob)$, the $V_i$'s taking values in
$[0,+\infty)$, and for each time $t\ge 1$ denote by $\Psi_t$ the
  expected value
\begin{equation}
\Psi_t:=\Ex\bigg[U_t\prod_{i\ge 1}\Big(\mathds{1}_{\{T_i>t\}}+V_i\mathds{1}_{\{T_i\le t\}}\Big)\bigg].
\label{def:Zphi}
\end{equation}
Here we determine the asymptotic exponential rate of growth of
$\Psi_t$ with respect to $t$. The solution to this problem is a needed
preliminary step to relate the rate function $J$ to the function $z$
defined by (\ref{defzetafun}).  The computation of $\Psi_t$ takes
advantage of the {\it generalized renewal equation}
\begin{equation}
\Psi_t=\sum_{s=1}^t a_s\Psi_{t-s}
\label{gen_renewaleq}
\end{equation}
satisfied for each $t\ge 1$ with the initial condition $\Psi_0:=1$,
where $a_s:=\Ex[V_1\mathds{1}_{\{S_1=s\}}]$ is a non-negative extended real
number. This equation is deduced conditioning on $S_1$ and then using
the fact that the renewal process starts over at the renewal time
$T_1$.  We are only interested in the case where $a_{\sigma_l}>0$ for
each $l$, $\sigma_1,\ldots,\sigma_m$ being the $m$ coprime integers
introduced in Section \ref{sec:basics} to make effective aperiodicity
of the waiting time distribution.

The expected value $A(\zeta):=\Ex[V_1e^{-\zeta
    S_1}\mathds{1}_{\{S_1<\infty\}}]=\sum_{s\ge 1}a_s e^{-\zeta s}$ exists
as an extended real number and defines a lower semicontinuous function
$A$ that maps $\zeta\in\Rl$ in $A(\zeta)$. The
number $\psi$ given by
\begin{equation}
\psi:=\inf\Big\{\zeta\in\Rl\,:\,A(\zeta)\le 1\Big\},
\label{zeta_o}
\end{equation}
where the infimum over the empty set is customarily interpreted as
$+\infty$, exactly is the exponential rate of growth we are looking for
as stated by the next proposition.  The level set
$\{\zeta\in\Rl:A(\zeta)\le 1\}$ is bounded from below since
$A(\zeta)\ge\sum_{l=1}^ma_{\sigma_l} e^{-\zeta \sigma_l}>1$ for all
$\zeta$ sufficiently negative and closed due to lower
semicontinuity. Consequently, $\psi>-\infty$ and $A(\psi)\le 1$ if
$\psi<+\infty$. It follows that $\Psi_t\le e^{\psi t}$ for all $t\ge
1$, which is trivial if $\psi=+\infty$ and is easily verified by
induction starting from (\ref{gen_renewaleq}) when $\psi<+\infty$.
\begin{proposition}
  \label{Zlim}
$\lim_{t\uparrow\infty}(1/t)\ln \Psi_t$ exists as an extended real
  number and is equal to $\psi>-\infty$. Moreover, the bound $\Psi_t\le
  e^{\psi t}$ holds for all $t\ge 1$.
\end{proposition}

\begin{proof}
The bound $\Psi_t\le e^{\psi t}$ for all $t\ge 1$ gives
$\limsup_{t\uparrow\infty}(1/t)\ln \Psi_t\le \psi$. Let us show that
\begin{equation}
\liminf_{t\uparrow\infty}\frac{1}{t}\ln \Psi_t\ge \psi.
\label{theorem:Zlim1}
\end{equation}
We have $\Psi_t\ge\Ex[U_t\prod_{i\ge
    1}(\mathds{1}_{\{T_i>t\}}+V_i\mathds{1}_{\{T_i\le
    t\}})\prod_{i=1}^n\mathds{1}_{\{S_i=s_i\}}]=\prod_{i=1}^na_{s_i}$
if $t=\sum_{i=1}^ns_i$ . This way, the same arguments used in Section
\ref{sec:basics} to deduce $Z_t^c>0$ for all $t>t_c$ yield $\Psi_t>0$
for all $t>t_c$ as $a_{\sigma_l}>0$ by hypothesis for each $l$.  This
property allows us to prove (\ref{theorem:Zlim1}) as follows.  Pick a
real number $\zeta<\psi$ and notice that there exists an integer
$\tau\ge 1$ so that $\sum_{s=1}^\tau a_s e^{-\zeta s}\ge 1$. On the
contrary we would have $A(\zeta)\le 1$, which contradicts the
assumption that $\zeta<\psi$. Since $\Psi_t>0$ for all $t>t_c$, we can
find a constant $M>-\infty$ such that $\ln \Psi_t\ge M+\zeta t$ for
every $t$ satisfying $t_c<t\le t_c+\tau$. As a matter of fact, this
bound is valid for all $t>t_c$. Indeed, an argument by induction based
on the generalized renewal equation (\ref{gen_renewaleq}) shows that
if $t>t_c+\tau$ and $\ln \Psi_{t-s}\ge M+\zeta(t-s)$ for any positive
$s\le\tau$, then
\begin{equation*}
\Psi_t=\sum_{s=1}^ta_s\Psi_{t-s}\ge\sum_{s=1}^\tau a_s\Psi_{t-s}\ge e^{M+\zeta t}\,\sum_{s=1}^\tau a_s e^{-\zeta s}\ge e^{M+\zeta t}.
\end{equation*}
It follows that $\liminf_{t\uparrow\infty}(1/t)\ln \Psi_t\ge \zeta$, giving
(\ref{theorem:Zlim1}) once $\zeta$ is sent to $\psi$.
\end{proof}

The first application of proposition \ref{Zlim} we consider is
concerned with the function $z$ defined by (\ref{defzetafun}). To this
aim we pick a linear functional $\varphi\in\rew^\star$ and we set
$V_i:=e^{\varphi(X_i)+v(S_i)}$ for every $i$. In this case, we have
$a_{\sigma_l}=\Ex[e^{\varphi(X_1)+v(S_1)}\mathds{1}_{\{S_1=\sigma_l\}}]>0$
for each $l$ as $p(\sigma_l)>0$ and
$\Psi_t=\Ex[U_te^{\varphi(W_t)+H_t}]$ for all $t$ since
\begin{equation*}
\prod_{i\ge
  1}\Big[\mathds{1}_{\{T_i>t\}}+e^{\varphi(X_i)+v(S_i)}\mathds{1}_{\{T_i\le
t\}}\Big]=e^{\sum_{i\ge 1}[\varphi(X_i)+v(s_i)]\mathds{1}_{\{T_i\le t\}}}=e^{\varphi(W_t)+H_t}.
\end{equation*}
Moreover, a direct comparison with (\ref{defzetafun}) shows that the
number $\psi$ associated with the present $V_1$ by formula
(\ref{zeta_o}) is nothing but $z(\varphi)$. Consequently, proposition
\ref{Zlim} gives
$\lim_{t\uparrow\infty}(1/t)\ln\Ex[U_te^{\varphi(W_t)+H_t}]=z(\varphi)$
and $z(\varphi)>-\infty$.  It follows from here thanks to the
arbitrariness of $\varphi$ that $z$ is convex and that $z$ never
attains $-\infty$, thus resulting in a proper convex function since
$z$ is finite at least in 0 due to assumption \ref{ass2} as we have
seen at the beginning of Section \ref{par:mainres}. Proposition
\ref{Zlim} also shows that $\Ex[U_te^{\varphi(W_t)+H_t}]\le
e^{z(\varphi)t}$ for all $t\ge 1$. The function $z$ is lower
semicontinuous because if $\{\varphi_i\}_{i\ge 0}$ is a sequence
converging to $\varphi$ and $t$ is any positive integer, then the
bound $e^{z(\varphi_i)t}\ge\Ex[U_te^{\varphi_i(W_t)+H_t}]$ and Fatou's
lemma give $\liminf_{i\uparrow\infty}z(\varphi_i)\ge
\liminf_{i\uparrow\infty}(1/t)\ln\Ex[U_te^{\varphi_i(W_t)+H_t}]\ge
(1/t)\ln\Ex[U_te^{\varphi(W_t)+H_t}]$, which results in
$\liminf_{i\uparrow\infty}z(\varphi_i)\ge z(\varphi)$ when $t$ is sent
to infinity. We have thus proved the following lemma.
\begin{lemma}
\label{zprop}
  The function $z$ is proper convex and lower semicontinuous. Given
  any $\varphi\in\rew^\star$, the bound
  $\Ex[U_te^{\varphi(W_t)+H_t}]\le e^{z(\varphi)t}$ is valid for all
  $t\ge 1$ and the limit
  $\lim_{t\uparrow\infty}(1/t)\ln\Ex[U_te^{\varphi(W_t)+H_t}]=z(\varphi)$
  holds.
\end{lemma}

\subsection{Connection with the function $z$}
\label{par:cI}

In this section we prove that the rate function $J$ is the
Fenchel-Legendre transform of $z$, namely that
$J(w)=\sup_{\varphi\in\rew^\star}\{\varphi(w)-z(\varphi)\}$ for all
$w\in\rew$. Lemma \ref{Jprop} states that $J$ is convex and lower
semicontinuous. Actually, $J$ is proper convex. Indeed, by combining
lemma \ref{start} with $C:=\rew$ and lemma \ref{zprop} with
$\varphi:=0$ we get
$\mathcal{L}(\rew)=\lim_{t\uparrow\infty}(1/t)\ln\mu_t(\rew)=\lim_{t\uparrow\infty}(1/t)\ln\Ex[U_te^{H_t}]=z(0)$.
This way, part (i) of proposition \ref{WLDP} with $G:=\rew$ gives
$z(0)\ge-\inf_{w\in\rew}\{J(w)\}$ and lemma \ref{lemma_conv} with
$C:=\rew$ yields $z(0)\le-\inf_{w\in\rew}\{J(w)\}$, with the
consequence that $\inf_{w\in\rew}\{J(w)\}=-z(0)$. As $z(0)$ is finite,
this equality shows that $J$ is finite at some point and that it never
attains $-\infty$. Proper convexity and lower semicontinuity allow us
to express $J$ in terms of its convex conjugate $J^\star$ as follows
(see \cite{Zalinescu}, theorem 2.3.3):
\begin{equation}
\label{Jstar}
J(w)=\sup_{\varphi\in\rew^\star}\Big\{\varphi(w)-J^\star(\varphi)\Big\}
\end{equation}
  for every $w\in\rew$ with
  $J^\star(\varphi):=\sup_{w\in\rew}\{\varphi(w)-J(w)\}$ for all
  $\varphi\in\rew^\star$.  This way, in order to demonstrate that $J$
  is the Fenchel-Legendre transform of $z$ it suffices to show that
  $J^\star=z$. Basically, this argument is the same argument used by
  Cerf and Petit \cite{Petit} for a short proof of Cram\'er's theorem
  in $\Rl$.

Proving the bound $J^\star(\varphi)\le z(\varphi)$ for all
$\varphi\in\rew^\star$ is not difficult.  To do this, we fix
$\varphi\in\rew^\star$ and we observe that lemma \ref{zprop} together
with the fact that $\varphi(W_t-tw)\ge-\|W_t-tw\|\|\varphi\|\ge
-t\delta\|\varphi\|$ if $W_t/t\in B_{w,\delta}$ gives for every $t\ge
1$, $w\in\rew$, and $\delta>0$
\begin{eqnarray}
  \nonumber
  e^{z(\varphi) t}\ge\Ex\Big[U_t e^{\varphi(W_t)+H_t}\Big]
  &\ge&\Ex\bigg[\mathds{1}_{\big\{\frac{W_t}{t}\in B_{w,\delta}\big\}}U_t e^{\varphi(W_t)+H_t}\bigg]\\
  \nonumber
  &=& e^{t\varphi(w)}\,\Ex\bigg[\mathds{1}_{\big\{\frac{W_t}{t}\in B_{w,\delta}\big\}}U_t e^{\varphi(W_t-tw)+H_t}\bigg]\\
  \nonumber
  &\ge&e^{t\varphi(w)-t\delta\|\varphi\|}\,\Ex\bigg[\mathds{1}_{\big\{\frac{W_t}{t}\in B_{w,\delta}\big\}}U_t e^{H_t}\bigg]\\
\nonumber
  &=&e^{t\varphi(w)-t\delta\|\varphi\|}\,\mu_t(B_{w,\delta}).
\end{eqnarray}
Taking logarithms, dividing by $t$, and sending $t$ to infinity, we
get from here
$z(\varphi)\ge\varphi(w)+\mathcal{L}(B_{w,\delta})-\delta\|\varphi\|\ge
\varphi(w)-J(w)+\delta\|\varphi\|$ thanks to lemma \ref{start}. Thus,
sending $\delta$ to zero first and appealing to the arbitrariness of
$w$ later we reach the bound
$z(\varphi)\ge\sup_{w\in\rew}\{\varphi(w)-J(w)\}=:J^\star(\varphi)$. A
more sophisticated use of proposition \ref{Zlim} leads to the opposite
bound, and hence to equality as stated by the following proposition.
\begin{proposition}
  \label{zeta_J}
  The convex conjugate $J^\star$ of $J$ equals $z$.
\end{proposition}

\begin{proof}
  Pick a linear functional $\varphi\in\rew^\star$. As $z(\varphi)\ge
  J^\star(\varphi)$, in order to show that
  $z(\varphi)=J^\star(\varphi)$ we must prove that $z(\varphi)\le
  J^\star(\varphi)$. Assume that $J^\star(\varphi)<+\infty$, otherwise
  there is nothing to prove. We are going to obtain the bound
  $z(\varphi)\le J^\star(\varphi)$ in two steps. At first we verify
  that for each $K\subseteq\rew$ compact
\begin{equation}
\limsup_{t\uparrow\infty}\frac{1}{t}\ln\Ex\bigg[\mathds{1}_{\big\{\frac{W_t}{t}\in K \big\}}U_t e^{\varphi(W_t)+H_t}\bigg]
\le J^\star(\varphi).
\label{zeta_J_1}
\end{equation}
Then, we demonstrate that for each real number $\zeta<z(\varphi)$
there exists a compact convex set $K\subseteq\rew$ with the property
that
\begin{equation}
\zeta<\limsup_{t\uparrow\infty}\frac{1}{t}\ln\Ex\bigg[\mathds{1}_{\big\{\frac{W_t}{t}\in K \big\}}U_te^{\varphi(W_t)+H_t}\bigg].
  \label{zeta_J_3}
\end{equation}
The proposition follows by combining (\ref{zeta_J_3}) with
(\ref{zeta_J_1}) first and by sending $\zeta$ to $z(\varphi)$ later.

Let us prove (\ref{zeta_J_1}) for a given compact set $K$ in $\rew$.
Let $\lambda>J^\star(\varphi)$ and $\rho>0$ be two real numbers. Since
there exists $\epsilon>0$ such that
$\varphi(w)+\inf_{\delta>0}\{\mathcal{L}(B_{w,\delta})\}=\varphi(w)-J(w)\le
J^\star(\varphi)\le \lambda-\epsilon$ for all $w$, for each $w\in\rew$
we can find $\delta_w>0$ in such a way that $\delta_w\|\varphi\|<\rho$
and $\mathcal{L}(B_{w,\delta_w})\le \lambda-\varphi(w)$.  Lemma
\ref{start} gives
$\lim_{t\uparrow\infty}(1/t)\ln\mu_t(B_{w,\delta_w})\le
\lambda-\varphi(w)$ for such $\delta_w$. Furthermore, we have
$\varphi(W_t-tw)\le\|W_t-tw\|\|\varphi\|\le
t\delta_w\|\varphi\|<t\rho$ if $W_t/t\in B_{w,\delta_w}$.  From the
compactness of $K$ there exist finitely many points $w_1,\ldots,w_n$
in $K$ so that $K\subseteq\cup_{i=1}^n B_{w_i,\delta_{w_i}}$. It
follows that for all $t\ge 1$
\begin{eqnarray}
\nonumber
\Ex\bigg[\mathds{1}_{\big\{\frac{W_t}{t}\in K \big\}}U_t e^{\varphi(W_t)+H_t}\bigg]&\le&
\sum_{i=1}^n\Ex\bigg[\mathds{1}_{\big\{\frac{W_t}{t}\in B_{w_i,\delta_{w_i}}\big\}}U_t e^{\varphi(W_t)+H_t}\bigg]\\
\nonumber
&=&\sum_{i=1}^ne^{t\varphi(w_i)}\,\Ex\bigg[\mathds{1}_{\big\{\frac{W_t}{t}\in B_{w_i,\delta_{w_i}}\big\}}U_t e^{\varphi(W_t-tw_i)+H_t}\bigg]\\
\nonumber
&\le&\sum_{i=1}^n\mu_t(B_{w_i,\delta_{w_i}})e^{t \varphi(w_i)+t\rho}.
\end{eqnarray}
By combining this bound with
$\lim_{t\uparrow\infty}(1/t)\ln\mu_t(B_{w_i,\delta_{w_i}})\le
\lambda-\varphi(w_i)$ for each $i$ we find
\begin{equation*}
\limsup_{t\uparrow\infty}\frac{1}{t}\ln\Ex\bigg[\mathds{1}_{\big\{\frac{W_t}{t}\in K \big\}}U_t e^{\varphi(W_t)+H_t}\bigg]\le \lambda +\rho.
\end{equation*}
This way, we reach (\ref{zeta_J_1}) by sending $\lambda$ to
$J^\star(\varphi)$ and $\rho$ to 0.

We now verify (\ref{zeta_J_3}). Pick a real number $\zeta<z(\varphi)$
and observe that necessarily $\Ex[e^{\varphi(X_1)+v(S_1)-\zeta
    S_1}\mathds{1}_{\{S_1<\infty\}}]>1$ by definition of
$z(\varphi)$. Recall that
$\Ex[e^{\varphi(X_1)+v(S_1)}\mathds{1}_{\{S_1=\sigma_l\}}]>0$ for all
$l$ since the $m$ coprime integers $\sigma_1,\ldots,\sigma_m$ satisfy
$p(\sigma_l)>0$ for every $l$. We shall show at the end that there
exists a compact convex set $K\subseteq\rew$ such that
\begin{equation}
\Ex\Big[e^{\varphi(X_1)+v(S_1)-\zeta S_1}\mathds{1}_{\{X_1/S_1\in K\}}\mathds{1}_{\{S_1<\infty\}}\Big]>1
\label{zeta_J_4}
\end{equation}
and 
\begin{equation}
\Ex\Big[e^{\varphi(X_1)+v(S_1)}\mathds{1}_{\{X_1/S_1\in K\}}\mathds{1}_{\{S_1=\sigma_l\}}\Big]>0
\label{zeta_J_5}
\end{equation}
for each $l$.  This way, setting
$V_i:=e^{\varphi(X_i)+v(S_i)}\mathds{1}_{\{X_i/S_i\in K\}}$ for all
$i$ and introducing the number $\psi$ defined by
\begin{equation*}
\psi:=\inf\bigg\{\eta\in\Rl\,:\,\Ex\Big[V_1e^{-\eta S_1}\mathds{1}_{\{S_1<\infty\}}\Big]\le 1\bigg\},
\end{equation*}
we have $\zeta<\psi$ from (\ref{zeta_J_4}). At the same time, if
$a_s:=\Ex[V_1\mathds{1}_{\{S_1=s\}}]$ for all $s$, then (\ref{zeta_J_5})
gives $a_{\sigma_l}>0$ for each $l$.  Consequently, we can invoke
proposition \ref{Zlim} with the present $V_i$ to get
\begin{eqnarray}
  \nonumber
  \zeta<\psi&=&\lim_{t\uparrow\infty}\frac{1}{t}\ln\Ex\bigg[U_t\prod_{i\ge 1}\Big(\mathds{1}_{\{T_i>t\}}+V_i\mathds{1}_{\{T_i\le t\}}\Big)\bigg]\\
  \nonumber
  &=&\lim_{t\uparrow\infty}\frac{1}{t}
  \ln\Ex\bigg[U_te^{\varphi(W_t)+H_t}\prod_{i\ge 1}\Big(\mathds{1}_{\{T_i>t\}}+\mathds{1}_{\{X_i/S_i\in K\}}\mathds{1}_{\{T_i\le t\}}\Big)\bigg].
\end{eqnarray}
On the other hand, as $K$ is convex, the condition $X_i/S_i\in K$ for
all $i$ such that $T_i\le t$ entails $W_t/t\in K$ when $t$ is a
renewal. To understand this point, we write $W_t/t=\sum_{i\ge
  1}(X_i/S_i)(S_i/t)\mathds{1}_{\{T_i\le t\}}$ and we notice that when
there exists a positive integer $n$ such that $T_n=t$, then
$\sum_{i\ge 1}(S_i/t)\mathds{1}_{\{T_i\le
  t\}}=\sum_{i=1}^n(S_i/t)=T_n/t=1$.  It follows that
\begin{eqnarray}
  \nonumber
  \zeta&<&\lim_{t\uparrow\infty}\frac{1}{t}
  \ln\Ex\bigg[U_te^{\varphi(W_t)+H_t}\prod_{i\ge 1}\Big(\mathds{1}_{\{T_i>t\}}+\mathds{1}_{\{X_i/S_i\in K\}}\mathds{1}_{\{T_i\le t\}}\Big)\bigg]\\
  \nonumber
  &=&\lim_{t\uparrow\infty}\frac{1}{t}
  \ln\Ex\bigg[U_te^{\varphi(W_t)+H_t}\,\mathds{1}_{\big\{\frac{W_t}{t}\in K \big\}}
    \prod_{i\ge 1}\Big(\mathds{1}_{\{T_i>t\}}+\mathds{1}_{\{X_i/S_i\in K\}}\mathds{1}_{\{T_i\le t\}}\Big)\bigg]\\
  \nonumber
  &\le&\limsup_{t\uparrow\infty}\frac{1}{t}\ln\Ex\bigg[U_te^{\varphi(W_t)+H_t}\,\mathds{1}_{\big\{\frac{W_t}{t}\in K \big\}}\bigg],
\end{eqnarray}
which proves (\ref{zeta_J_3}).

To conclude the proof of the proposition, we must show the validity of
(\ref{zeta_J_4}) and (\ref{zeta_J_5}) for some compact convex set
$K$. To this aim, consider the finite measure
$\pi_R:=\Ex[e^{\varphi(X_1)+v(S_1)-\zeta
    S_1}\mathds{1}_{\{X_1/S_1\in\cdot\}}\mathds{1}_{\{\|X_1\|\le
    R\}}\mathds{1}_{\{S_1\le R\}}]$ on $\Brew$, $R$ being a positive
real number. The fact that $\Ex[e^{\varphi(X_1)+v(S_1)-\zeta
    S_1}\mathds{1}_{\{S_1<\infty\}}]>1$ implies that there exists a
sufficiently large $R$ so that $\pi_R(\rew)>1$ and completeness and
separability of $\rew$ entail that $\pi_R$ is tight (see
\cite{Bogachevbook}, theorem 7.1.7). It follows that there exists a
compact set $K_o$ such that $\pi_R(K_o)>1$, which gives
$\Ex[e^{\varphi(X_1)+v(S_1)-\zeta S_1}\mathds{1}_{\{X_1/S_1\in
    K_o\}}\mathds{1}_{\{S_1<\infty\}}]\ge \pi_R(K_o)>1$. Similar
arguments with
$\pi_R:=\Ex[e^{\varphi(X_1)+v(S_1)}\mathds{1}_{\{X_1/S_1\in\cdot\}}\mathds{1}_{\{\|X_1\|\le
    R\}}\mathds{1}_{\{S_1=\sigma_l\}}]$ in combination with
$\Ex[e^{\varphi(X_1)+v(S_1)}\mathds{1}_{\{S_1=\sigma_l\}}]>0$ yield
$\Ex[e^{\varphi(X_1)+v(S_1)}\mathds{1}_{\{X_1/S_1\in
    K_l\}}\mathds{1}_{\{S_1=\sigma_l\}}]>0$ for some compact set
$K_l$.  Let $K$ be the closed convex hull of $K_o\cup
K_1\cup\cdots\cup K_m$. The set $K$ is convex and compact (see
\cite{Rudin}, theorem 3.20) and satisfies (\ref{zeta_J_4}) and
(\ref{zeta_J_5}) as $K_o\subseteq K$ and $K_l\subseteq K$ for each
$l$.
\end{proof}

\subsection{Proposition \ref{proposition:free_energy} and theorem \ref{mainth1} point by point}
\label{par:th1pp}

In this section we explicitly verify proposition
\ref{proposition:free_energy} and theorem \ref{mainth1} point by
point, but the former simply is part of lemma \ref{zprop} and does not
need other demonstrations.  Lemma \ref{Jprop} states that $J$ is
convex and lower semicontinuous. Moreover, we have seen that $J$ is
proper convex at the beginning of the last section. As
$J(w)=\sup_{\varphi\in\rew^\star}\{\varphi(w)-z(\varphi)\}$ for all
$w\in\rew$ thanks to (\ref{Jstar}) and proposition \ref{zeta_J}, the
rate function $I$ defined by (\ref{defIorate}) equals $J+z(0)$ and
inherits the lower semicontinuity and proper convexity of $J$. These
facts prove part (a) of theorem \ref{mainth1}.  Part (b) of theorem
\ref{mainth1} follows from part (i) of propositions \ref{WLDP} bearing
in mind that $\ln\prob_t[W_t/t\in \cdot\,\,]=\ln\mu_t-\ln Z_t^c$ for
each $t>t_c$, that $\lim_{t\uparrow\infty}(1/t)\ln
Z_t^c=\lim_{t\uparrow\infty}\Ex[U_te^{H_t}]=z(0)$ by lemma
\ref{zprop}, and that $I=J+z(0)$.  Similarly, part (c) of theorem
\ref{mainth1} concerning compact sets is due to part (ii) of
proposition \ref{WLDP}.  Part (c) regarding convex sets follows from
the limit $\limsup_{t\uparrow\infty}(1/t)\ln\mu_t(C)=\mathcal{L}(C)$
valid for any $C\in\Brew$ convex and lemma \ref{lemma_conv}.  Finally,
part (c) for closed sets in the finite-dimensional case is
demonstrated by the following proposition.  Let us observe that the
exponential moment condition $\Ex[e^{\xi\|X_1\|+v(S_1)-\zeta
    S_1}\mathds{1}_{\{S_1<\infty\}}]<+\infty$ for some numbers
$\zeta\ge 0$ and $\xi>0$ implies $z(\varphi)<+\infty$ for all
$\varphi\in\rew^\star$ such that $\|\varphi\|\le\xi$. Indeed, the
validity of this condition with certain $\zeta\ge 0$ and $\xi>0$
entails that a number $h$ large enough can be found so that
$e^{-h}\,\Ex[e^{\xi\|X_1\|+v(S_1)-\zeta
    S_1}\mathds{1}_{\{S_1<\infty\}}]\le 1$. It follows that
$\Ex[e^{\varphi(X_1)+v(S_1)-(\zeta+h)
    S_1}\mathds{1}_{\{S_1<\infty\}}]\le
e^{-h}\,\Ex[e^{\xi\|X_1\|+v(S_1)-\zeta
    S_1}\mathds{1}_{\{S_1<\infty\}}]\le 1$ if $\|\varphi\|\le\xi$ as
$S_1\ge 1$ with full probability, which gives $z(\varphi)<+\infty$
according to definition (\ref{defzetafun}). It is easy to verify that
if $\rew$ is finite-dimensional, then the above exponential moment
condition is tantamount to the existence of $\xi>0$ such that
$z(\varphi)<+\infty$ for all $\varphi$ fulfilling $\|\varphi\|\le\xi$.
\begin{proposition}
\label{closed_constrained}
Assume that $\rew$ has finite dimension and that there exist numbers
$\zeta\ge 0$ and $\xi>0$ such that $\Ex[e^{\xi\|X_1\|+v(S_1)-\zeta
    S_1}\mathds{1}_{\{S_1<\infty\}}]<+\infty$. Then, for each
$F\subseteq\rew$ closed
\begin{equation*}
  \limsup_{t\uparrow\infty}\frac{1}{t}\ln\mu_t(F)\le-\inf_{w\in F}\{J(w)\}.
\end{equation*}
\end{proposition}

\begin{proof}
  Fix a closed set $F$ in $\rew$ and observe that $\inf_{w\in
    F}\{J(w)\}\ge\inf_{w\in \rew}\{J(w)\}=-z(0)>-\infty$.  Then, pick
  a real number $\lambda<\inf_{w\in F}\{J(w)\}$.  Let $d$ be the
  dimension of $\rew$, let $\{w_1,\ldots,w_d\}$ be a basis of $\rew$,
  and let $\{\vartheta_1,\ldots,\vartheta_d\}\subset\rew^\star$ be the
  dual basis: $\vartheta_i(w_j)$ equals 1 if $i=j$ and 0 otherwise for
  all $i$ and $j$. For $i$ ranging from $1$ to $d$, set
  $\varphi_i:=\vartheta_i/\|\vartheta_i\|$ and
  $\varphi_{d+i}:=-\varphi_i$. Since $z(\varphi)<+\infty$ if
  $\|\varphi\|\le\xi$, $\xi>0$ being the number associated with the
  hypothesized exponential moment condition, there exists a real
  number $\rho>0$ with the property that $z(\xi\varphi_i)-\xi\rho\le
  -\lambda$ for each $i$.  Denoting by $K$ the compact set
  $K:=\cap_{i=1}^{2d}\{w\in\rew:\varphi_i(w)\le \rho\}$, we have
  $K^c=\cup_{i=1}^{2d}\{w\in\rew:\varphi_i(w)> \rho\}$.  This way, by
  making use of the Chernoff bound first and of the bound
  $\Ex[U_te^{\xi\varphi_i(W_t)+H_t}]\le e^{z(\xi\varphi_i)t}$ due to
  lemma \ref{zprop} later, we obtain
\begin{eqnarray}
  \nonumber
  \mu_t(K^c)&\le&\sum_{i=1}^{2d} \mu_t\Big(\big\{w\in\rew:\varphi_i(w)>\rho\big\}\Big)
  =\sum_{i=1}^{2d}\Ex\Big[\mathds{1}_{\{\varphi_i(W_t)>\rho t\}}U_te^{H_t}\Big]\\
  \nonumber
  &\le&\sum_{i=1}^{2d}\Ex\Big[U_te^{\xi\varphi_i(W_t)-\xi\rho t+H_t}\Big]\le\sum_{i=1}^de^{z(\xi\varphi_i) t-\xi\rho t}\le 2d e^{-\lambda t},
\end{eqnarray}
giving $\mu_t(F)=\mu_t(F\cap K)+\mu_t(F\cap K^c)\le\mu_t(F\cap K)+2d
e^{-\lambda t}$ for each $t$. On the other hand, part (ii) of
proposition \ref{WLDP} with the compact set $F\cap K$ shows that
$\limsup_{t\uparrow\infty}(1/t)\ln\mu_t(F\cap K)\le-\inf_{w\in F\cap
  K}\{J(w)\} \le-\inf_{w\in F}\{J(w)\}\le-\lambda$. It follows that
$\limsup_{t\uparrow\infty}(1/t)\ln\mu_t(F)\le-\lambda$, which proves
the proposition once $\lambda$ is sent to $\inf_{w\in F}\{J(w)\}$.
\end{proof}

\section{Proof of proposition \ref{proposition:freefree_energy} and of theorems \ref{mainth2} and \ref{mainth3}}
\label{sec:freemodels}

Large deviation bounds within the PM can be made a consequence of the
corresponding bounds in the CPM by exploiting conditioning as follows.
Pick an integer time $t\ge 1$ and notice that if $T_1\le t$, then
there is one and only one positive integer $n\le t$ such that $T_n\le
t$ and $T_{n+1}>t$.  Thus, $\Omega=\{T_1>t\}\cup\{T_1\le t\}$ and
$\{T_1\le t\}=\cup_{n=1}^t\{T_n\le t \mbox{ and }
T_{n+1}>t\}=\cup_{n=1}^t\cup_{\tau=n}^t\{T_n=\tau \mbox{ and }
T_{n+1}>t\}$, the events $\{T_n=\tau \mbox{ and }T_{n+1}>t\}$ for
$1\le n\le\tau\le t$ being disjoint. The condition $T_1>t$ is
tantamount to $S_1>t$ and implies that $H_t=0$ and $W_t=0$.  The
condition $T_n=\tau$ and $T_{n+1}>t$ is tantamount to $T_n=\tau$ and
$S_{n+1}>t-\tau$ and implies that $H_t=\sum_{i=1}^nv(S_i)=H_\tau$ and
$W_t=\sum_{i=1}^nX_i=W_\tau$ are independent of $S_{n+1}$. This way,
for every $\varphi\in\rew^\star$ we find the identity between measures
\begin{eqnarray}
  \nonumber
  \Ex\Big[\mathds{1}_{\{W_t\in \cdot\}}e^{\varphi(W_t)+H_t}\Big]
  &=&\Ex\Big[\mathds{1}_{\{W_t\in \cdot\}}\mathds{1}_{\{S_1>t\}}e^{\varphi(W_t)+H_t}\Big]\\
  \nonumber
  &+&\sum_{n=1}^t\sum_{\tau=n}^t\Ex\Big[\mathds{1}_{\{W_t\in\cdot\}}\mathds{1}_{\{T_n=\tau\}}\mathds{1}_{\{S_{n+1}>t-\tau\}}e^{\varphi(W_t)+H_t}\Big]\\
\nonumber
&=&\mathds{1}_{\{0\in\cdot\}}\cdot\prob[S_1>t]\\
\nonumber
&+&\sum_{\tau=1}^t\sum_{n=1}^\tau\Ex\Big[\mathds{1}_{\{W_\tau\in\cdot\}}\mathds{1}_{\{T_n=\tau\}}e^{\varphi(W_\tau)+H_\tau}\Big]\cdot\prob[S_1>t-\tau]\\
\nonumber
&=&\mathds{1}_{\{0\in\cdot\}}\cdot\prob[S_1>t]\\
&+&\sum_{\tau=1}^t\Ex\Big[\mathds{1}_{\{W_\tau\in\cdot\}}U_\tau e^{\varphi(W_\tau)+H_\tau}\Big]\cdot\prob[S_1>t-\tau].
\label{start_free}
\end{eqnarray}

Formula (\ref{start_free}) connects the free setting with the
constrained setting and is the starting point to prove proposition
\ref{proposition:freefree_energy} and theorems \ref{mainth2} and
\ref{mainth3}. Once again, we leave normalization aside at the
beginning and focus on the measure
$\nu_t:=Z_t\prob_t[W_t/t\in\cdot\,]$ on $\Brew$. Identity
(\ref{start_free}) with $\varphi=0$ results in the expression
\begin{eqnarray}
  \nonumber
  \nu_t&=&\Ex\bigg[\mathds{1}_{\big\{\frac{W_t}{t}\in \cdot\big\}}e^{H_t}\bigg]\\
  &=&\mathds{1}_{\{0\in\cdot\}}\cdot\prob[S_1>t]+
  \sum_{\tau=1}^t\Ex\bigg[\mathds{1}_{\big\{\frac{W_\tau}{t}\in \cdot\big\}}U_\tau e^{H_\tau}\bigg]\cdot\prob[S_1>t-\tau].
\label{start_free_sec}
\end{eqnarray}
We use this expression to derive a lower large deviation bound in
Section \ref{par:lowerfree} and an upper large deviation bound in
Section \ref{par:upperfree}. Theorem \ref{mainth2} is verified point
by point in Section \ref{th2pp}, where two counterexamples are also
shown to demonstrate that the upper large deviation bound for open
convex sets and closed convex sets cannot hold in general when
$\ells=-\infty$ and $I(0)=+\infty$. Finally, theorem \ref{mainth3} is
verified point by point in Section \ref{th3pp}.

Regarding proposition \ref{proposition:freefree_energy}, it is an
immediate consequence of formula (\ref{start_free}), which entails
$\Ex[e^{\varphi(W_t)+H_t}]=\prob[S_1>t]+\sum_{\tau=1}^t\Ex[U_\tau
  e^{\varphi(W_\tau)+H_\tau}]\cdot\prob[S_1>t-\tau]$ for each $t$ and
$\varphi$. Indeed, recalling the definitions
$\liminf_{t\uparrow\infty}(1/t)\ln\prob[S_1>t]=:\elli$ and
$\limsup_{t\uparrow\infty}(1/t)\ln\prob[S_1>t]=:\ells$, as well as the
limit $\lim_{t\uparrow\infty}(1/t)\ln \Ex[U_t
  e^{\varphi(W_t)+H_t}]=z(\varphi)$ due to lemma \ref{zprop}, for all
$\varphi\in\rew^\star$ we get
\begin{equation}
  \liminf_{t\uparrow\infty}\frac{1}{t}\ln \Ex\big[e^{\varphi(W_t)+H_t}\big]\ge z(\varphi)\vee\elli
  \label{Zt_inf}
\end{equation}
and
\begin{equation}
  \limsup_{t\uparrow\infty}\frac{1}{t}\ln \Ex\big[e^{\varphi(W_t)+H_t}\big]\le z(\varphi)\vee\ells.
  \label{Zt_sup}
\end{equation}

\subsection{The lower large deviation bound}
\label{par:lowerfree}

In this section we prove the following lower bound without
restrictions on $\elli$ and $\ells$.
\begin{proposition}
\label{free_open_nonorm}
  For each $G\subseteq\rew$ open
\begin{equation*}
  \liminf_{t\uparrow\infty}\frac{1}{t}\ln\nu_t(G)\ge-\adjustlimits\inf_{w\in G}\sup_{\varphi\in\rew^\star}
  \Big\{\varphi(w)-z(\varphi)\vee\elli\Big\}.
\end{equation*}
\end{proposition}

\begin{proof}
Pick an open set $G$ in $\rew$.  In order to demonstrate the
proposition it suffices to verify that for all $w\in G$
\begin{equation}
  \liminf_{t\uparrow\infty}\frac{1}{t}\ln \nu_t(G)\ge
  -\sup_{\varphi\in\rew^\star}\Big\{\varphi(w)-z(\varphi)\vee\elli\Big\}.
\label{lower_free_1}
\end{equation}
This bound is immediate when $\elli=-\infty$.  Indeed, keeping only
the term corresponding to $\tau=t$ in the r.h.s.\ of
(\ref{start_free_sec}) we get $\nu_t(G)\ge\mu_t(G)$, which shows that
$\liminf_{t\uparrow\infty}(1/t)\ln \nu_t(G)\ge-J(w)$ for any $w\in G$
thanks to part (i) of proposition \ref{WLDP}. On the other hand,
$J(w)$ is the r.h.s.\ of (\ref{lower_free_1}) if $\elli=-\infty$ by
formula (\ref{Jstar}) and proposition \ref{zeta_J}.

The proof of (\ref{lower_free_1}) is more laborious when
$\elli>-\infty$ and we assume that $\elli>-\infty$ from now on.  Let
$\dom z:=\{\varphi\in\rew^\star:z(\varphi)<+\infty\}$ be the effective
domain of $z$ and consider the function $F$ that for a given $w\in G$
maps $(\beta,\varphi)\in[0,1]\times\dom z$ in the real number
$F(\beta,\varphi):=\varphi(w)-\beta z(\varphi)-(1-\beta)\elli$. The
function $F$ is concave and upper semicontinuous with respect to
$\varphi$ for each fixed $\beta\in[0,1]$, inheriting these properties
from $z$, and convex and continuous with respect to $\beta$ for each
fixed $\varphi\in\dom z$. Then, due to compactness of the closed
interval $[0,1]$, Sion's minimax theorem allows us to exchange the
infimum over $\beta\in[0,1]$ and the supremum over $\varphi\in\dom z$:
$\sup_{\varphi\in\scriptsize{\dom}z}\inf_{\beta\in[0,1]}\big\{F(\beta,\varphi)\big\}=
\inf_{\beta\in[0,1]}\sup_{\varphi\in\scriptsize{\dom}z}\big\{F(\beta,\varphi)\big\}$.
As $\inf_{\beta\in[0,1]}\{\varphi(w)-\beta
z(\varphi)-(1-\beta)\elli\}=\varphi(w)-z(\varphi)\vee\elli$ and
$z(\varphi)\vee\elli=+\infty$ when $\varphi\notin\dom z$, this
identity can be written as
\begin{equation*}
  \sup_{\varphi\in\rew^\star}\Big\{\varphi(w)-z(\varphi)\vee\elli\Big\}
  =\adjustlimits\inf_{\beta\in[0,1]}\sup_{\varphi\in\scriptsize{\dom}z}\Big\{\varphi(w)-\beta z(\varphi)-(1-\beta)\elli\Big\}.
\end{equation*}
This way, we get the bound (\ref{lower_free_1}) if we prove that for
every $w\in G$ and $\beta\in[0,1]$
\begin{equation}
  \liminf_{t\uparrow\infty}\frac{1}{t}\ln \nu_t(G)\ge
  -\sup_{\varphi\in\scriptsize{\dom}z}\Big\{\varphi(w)-\beta z(\varphi)-(1-\beta)\elli\Big\}.
\label{lower_free_11}
\end{equation}

We prove (\ref{lower_free_11}) considering the case $\beta>0$ first.
Pick a point $w\in G$ and a number $\beta\in(0,1]$ and denote by
$\tau_t$ the greatest integer that is less than or equal to $\beta
t$. Let $\delta>0$ be such that $B_{w,2\delta}\subseteq G$ and focus
on all those sufficiently large integers $t$ such that $\tau_t>0$ and
$\|w\|<\beta\delta t$. Within this setting, we have that the event
$W_{\tau_t}/{\tau_t}\in B_{w/\beta,\delta}$ implies $W_{\tau_t}/t\in
B_{w,2\delta}\subseteq G$. Indeed, since $0\le t-\tau_t/\beta<1/\beta$
and $\|w\|<\beta\delta t$ we find
$\|W_{\tau_t}-tw\|\le\|W_{\tau_t}-(\tau_t/\beta)
w\|+(t-\tau_t/\beta)\|w\|<\|W_{\tau_t}-(\tau_t/\beta) w\|+\delta
t$. It follows that if $\|W_{\tau_t}-(\tau_t/\beta)w\|<\delta \tau_t$,
then $\|W_{\tau_t}-tw\|<\delta\tau_t+\delta t\le 2\delta t$.  This
way, keeping only the term corresponding to $\tau=\tau_t>0$ in the
r.h.s.\ of (\ref{start_free_sec}), we obtain
\begin{eqnarray}
\nonumber
\nu_t(G)&\ge&\Ex\bigg[\mathds{1}_{\big\{\frac{W_{\tau_t}}{t}\in G\big\}}U_{\tau_t}e^{H_t}\bigg]
\cdot\prob\big[S_1>t-\tau_t\big]\\
\nonumber
&\ge&\Ex\bigg[\mathds{1}_{\big\{\frac{W_{\tau_t}}{\tau_t}\in B_{w/\beta,\delta}\big\}}U_{\tau_t}e^{H_t}\bigg]\cdot\prob\big[S_1>t-\tau_t\big]\\
&=&\mu_{\tau_t}\Big(B_{\frac{w}{\beta},\delta}\Big)\cdot\prob\big[S_1>t-\tau_t\big].
\label{lower_free_12345}
\end{eqnarray}
We have
$\lim_{t\uparrow\infty}(1/\tau_t)\ln\mu_{\tau_t}(B_{w/\beta,\delta})=\mathcal{L}(B_{w/\beta,\delta})\ge-J(w/\beta)$
by lemma \ref{start}. We also have
$\lim_{t\uparrow\infty}\tau_t/t=\beta$ and
$\liminf_{t\uparrow\infty}(1/t)\ln\prob[S_1>t-\tau_t]=(1-\beta)\elli$. The
latter limit is trivial in the case $\beta=1$ to which $\tau_t=t$
corresponds, whereas it follows from
$\liminf_{t\uparrow\infty}(1/t)\ln\prob[S_1>t]=:\elli$ when $\beta<1$
due to the fact that $t-\tau_t$ is now diverging as $t$ is sent to
infinity. These arguments in combination with (\ref{lower_free_12345})
prove that
\begin{eqnarray}
\nonumber
\liminf_{t\uparrow\infty}\frac{1}{t}\ln \nu_t(G)&\ge&
-\beta J(w/\beta)+(1-\beta)\elli\\
\nonumber  
&=&-\sup_{\varphi\in\rew^\star}\Big\{\varphi(w)-\beta z(\varphi)\Big\}+(1-\beta)\elli\\
\nonumber
&=&-\sup_{\varphi\in\scriptsize{\dom}z}\Big\{\varphi(w)-\beta z(\varphi)-(1-\beta)\elli\Big\},
\end{eqnarray}
which is (\ref{lower_free_11}) under the hypothesis that $\beta>0$.

In order to settle the case $\beta=0$, we take a point $u\in\rew$ such
that $c:=J(u)$ is finite, which exists because $J$ is proper
convex. We have
$z(\varphi)=J^\star(\varphi):=\sup_{w\in\rew}\{\varphi(w)-J(w)\}\ge\varphi(u)-c$
for all $\varphi\in\rew^\star$ by proposition \ref{zeta_J}. As $G$ is
open, for a given $w\in G$ we can find a number $\delta\in(0,1)$ such
that $w+\epsilon u\in G$ whenever $\epsilon\in(0,\delta)$. Then, the
bound (\ref{lower_free_11}) applies with a positive
$\epsilon<\delta<1$ in place of $\beta$ and $w+\epsilon u$ in place of
$w$ to give
\begin{eqnarray}
\nonumber
\liminf_{t\uparrow\infty}\frac{1}{t}\ln \nu_t(G)&\ge&
-\sup_{\varphi\in \scriptsize{\dom}z}\Big\{\varphi(w+\epsilon u)-\epsilon z(\varphi)-(1-\epsilon)\elli\Big\}\\
  \nonumber
  &\ge&-\sup_{\varphi\in \scriptsize{\dom}z}\big\{\varphi(w)-\elli\big\}-\epsilon(c+\elli).
\end{eqnarray}
We obtain (\ref{lower_free_11}) corresponding to $\beta=0$ from here
by sending $\epsilon$ to zero.
\end{proof}

\subsection{The upper large deviation bound}
\label{par:upperfree}

An upper large deviation bound for compact sets can be proved by means
of standard arguments from large deviation theory without
distinguishing the case $\ells>-\infty$ from the case $\ells=-\infty$.
The following result holds.
\begin{proposition}
\label{free_compact}
  For each compact set $K\subseteq\rew$ 
\begin{equation*}
  \limsup_{t\uparrow\infty}\frac{1}{t}\ln\nu_t(K)\le-\adjustlimits\inf_{w\in K}\sup_{\varphi\in\rew^\star}
  \Big\{\varphi(w)-z(\varphi)\vee\ells\Big\}.
\end{equation*}
\end{proposition}

\begin{proof}
  Let $K$ be a compact set in $\rew$ and notice that $\inf_{w\in
    K}\sup_{\varphi\in\rew^\star}\{\varphi(w)-z(\varphi)\vee\ells\}\ge-z(0)\vee\ells>-\infty$. Let
  $\lambda<\inf_{w\in
    K}\sup_{\varphi\in\rew^\star}\{\varphi(w)-z(\varphi)\vee\ells\}$
  and $\rho>0$ be real numbers. As there exists $\epsilon>0$ such that
  $\sup_{\varphi\in\rew^\star}\{\varphi(w)-z(\varphi)\vee\ells\}\ge\lambda+\epsilon$
  for all $w\in K$, a linear functional $\varphi_w\in\rew^\star$ can
  be found for each $w\in K$ with the property that
  $\varphi_w(w)-z(\varphi_w)\vee\ells\ge\lambda$. It is manifest
  that $z(\varphi_w)<+\infty$ for such $\varphi_w$.  Let $\delta_w>0$
  be a number that satisfies $\delta_w\|\varphi_w\|\le\rho$. Then, for
  every positive integers $t$ and $\tau\le t$ the condition
  $W_\tau/t\in B_{w,\delta_w}$ entails
  $\varphi_w(W_\tau-tw)\ge-\|W_\tau-t
  w\|\|\varphi_w\|>-\delta_w\|\varphi_w\|t\ge-\rho t$, namely
  $\varphi_w(W_\tau)-t\varphi_w(w)+\rho t\ge 0$.  This way, bearing in
  mind that $\Ex[U_\tau e^{\varphi_w(W_\tau)+H_\tau}]\le
  e^{z(\varphi_w)\tau}$ by lemma \ref{zprop} we get for each $w\in K$
  and integers $t$ and $\tau\le t$
\begin{eqnarray}
\nonumber
  \Ex\bigg[\mathds{1}_{\big\{\frac{W_\tau}{t}\in B_{w,\delta_w}\big\}}U_\tau e^{H_\tau}\bigg]
  &\le&\Ex\Big[U_\tau e^{\varphi_w(W_\tau)-t\varphi_w(w)+t\rho+H_\tau}\Big]\\
  \nonumber
  &\le&e^{z(\varphi_w)\tau-t\varphi_w(w)+t\rho}\\
  &\le&e^{\tau[z(\varphi_w)\vee\ells]-t\varphi_w(w)+t\rho}.
\label{upper_free_1}
\end{eqnarray}
We also have for each $w\in K$ and $t$
\begin{equation}
\mathds{1}_{\{0\in B_{w,\delta_w}\}}\le e^{-t\varphi_w(w)+t\rho}
\label{upper_free_222}
\end{equation}
because if $0\in B_{w,\delta_w}$, then $\|w\|<\delta_w$ so that
$\varphi_w(w)\le\delta_w\|\varphi_w\|\le\rho$.

Due to the compactness of $K$, there exist finitely many points
$w_1,\ldots,w_n$ in $K$ such that $K\subseteq\cup_{i=1}^n
B_{w_i,\delta_{w_i}}$. The facts that
$\limsup_{t\uparrow\infty}(1/t)\ln\prob[S_1>t]=:\ells$ and
$z(\varphi_{w_i})\vee\ells>-\infty$ for each $i$ ensure the
existence of a positive constant $M<+\infty$ such that for all $t$ and
$i\le n$
\begin{equation}
\prob[S_1>t]\le Me^{t [z(\varphi_{w_i})\vee\ells]+t\rho}.
\label{upper_free_123}
\end{equation}
At this point, identity (\ref{start_free_sec}) combined with
(\ref{upper_free_1}), (\ref{upper_free_222}), and
(\ref{upper_free_123}) shows that for every $t$
\begin{eqnarray}
  \nonumber
  \nu_t(K)&\le&\sum_{i=1}^n\mathds{1}_{\{0\in B_{w_i,\delta_{w_i}}\}}\cdot\prob[S_1>t]\\
  \nonumber
  &+&\sum_{i=1}^n\sum_{\tau=1}^t\Ex\bigg[\mathds{1}_{\big\{\frac{W_\tau}{t}\in B_{w_i,\delta_{w_i}}\big\}}U_\tau e^{H_\tau}\bigg]\cdot
\prob\big[S_1>t-\tau\big]\\
\nonumber
&\le&M\sum_{i=1}^n\sum_{\tau=0}^te^{\tau [z(\varphi_{w_i})\vee\ells]-t\varphi_{w_i}(w_i)+t\rho}\cdot e^{(t-\tau)[z(\varphi_{w_i})\vee\ells]+(t-\tau)\rho}\\
\nonumber
&\le&M\sum_{i=1}^n\sum_{\tau=0}^te^{t [z(\varphi_{w_i})\vee\ells]-t\varphi_{w_i}(w_i)+2t\rho}
\le Mn(t+1)e^{-t\lambda+2t\rho},
\end{eqnarray}
which in turn yields $\limsup_{t\uparrow\infty}(1/t)\ln\nu_t(K)\le
-\lambda+2\rho$.  The proposition follows from here by sending $\rho$
to zero and $\lambda$ to $\inf_{w\in
  K}\sup_{\varphi\in\rew^\star}\{\varphi(w)-z(\varphi)\vee\ells\}$.
\end{proof}

The upper bound stated by proposition \ref{free_compact} cannot be
extended in general to convex sets when $\ells=-\infty$. However, at
least the following weaker upper bound holds for them.
\begin{lemma}
  \label{convex_free}
Let $C\subseteq\rew$ be open convex, closed convex, or any convex set
in $\Brew$ when $\rew$ is finite-dimensional.  Then, for each real
number $\ell\ge\ells$
\begin{equation*}
  \limsup_{t\uparrow\infty}\frac{1}{t}\ln\nu_t(C)\le
  -\adjustlimits\inf_{w\in C}\sup_{\varphi\in\rew^\star}\Big\{\varphi(w)
  -z(\varphi)\vee\ell\Big\}.
\end{equation*}
\end{lemma}

\begin{proof}
  Pick a real number $\ell\ge\ells$ and notice that $\inf_{w\in
    C}\sup_{\varphi\in\rew^\star}\{\varphi(w)-z(\varphi)\vee\ell\}\ge-z(0)\vee\ell>-\infty$. Fix
  a real number $\lambda<\inf_{w\in
    C}\sup_{\varphi\in\rew^\star}\{\varphi(w)-z(\varphi)\vee\ell\}$.
  To begin with, we observe that for any given real number $\eta\ge 1$
  and integer $\tau\ge 1$ we have the bound
\begin{equation}
\ln\mu_\tau(\eta C)\le -\lambda\eta\tau-\ell(\eta-1)\tau,
\label{setD_1}
\end{equation}
where $\eta C:=\{\eta w:w\in C\}\in\Brew$, which is convex, open if
$C$ is open, and closed if $C$ is closed.  Indeed, as there exists
$\epsilon>0$ such that
$\lambda+\epsilon\le\sup_{\varphi\in\rew^\star}\{\varphi(w)-z(\varphi)\vee\ell\}$
for all $w\in C$, for every $w\in C$ we can find
$\varphi_w\in\rew^\star$ satisfying $\lambda\le
\varphi_w(w)-z(\varphi_w)\vee\ell$. This way, for each $w\in C$
we obtain
\begin{eqnarray}
  \nonumber J(\eta w)&=&\sup_{\varphi\in\rew^\star}\big\{\varphi(\eta w)-z(\varphi)\big\}\ge \eta\varphi_w(w)-z(\varphi_w)\\
  \nonumber
  &\ge&\eta\lambda+\eta\big[z(\varphi_w)\vee\ell\big]-z(\varphi_w)\ge\eta\lambda+(\eta-1)\ell.
\end{eqnarray}
On the other hand, if $\gamma$ is a large enough integer so that
$\gamma\tau>t_c$, then the convexity of $\eta C$ allows us to invoke
super-additive properties to obtain $(1/\tau)\ln\mu_\tau(\eta
C)\le(1/\gamma\tau)\ln\mu_{\gamma\tau}(\eta C)\le\mathcal{L}(\eta C)$.
Consequently, lemma \ref{lemma_conv} with the set $\eta C$ entails
$\ln\mu_\tau(\eta C)\le-\tau\inf_{v\in \eta C}\{J(v)\}$, which proves
(\ref{setD_1}) because $\inf_{v\in \eta C}\{J(v)\}=\inf_{w\in
  C}\{J(\eta w)\}\ge \lambda\eta+\ell(\eta-1)$.

We use the bound (\ref{setD_1}) as follows. Given any positive
integers $t$ and $\tau\le t$, setting $\eta:=t/\tau$ we have that
$W_\tau/\tau\in \eta C$ is tantamount to $W_\tau/t\in C$. This way,
(\ref{setD_1}) yields
\begin{equation}
\Ex\bigg[\mathds{1}_{\big\{\frac{W_\tau}{t}\in C\big\}}U_\tau e^{H_\tau}\bigg]=\mu_\tau(\eta C)\le e^{-\lambda t-\ell(t-\tau)}.
\label{setD_2}
\end{equation}
For each $t$ we also find
\begin{equation}
\mathds{1}_{\{0\in C\}}\le e^{-\lambda t-\ell t}
\label{setD_234}
\end{equation}
because if $0\in C$, then
$\lambda<\sup_{\varphi\in\rew^\star}\{\varphi(w)-z(\varphi)\vee\ell\}$
with $w=0$ gives
$\lambda\le\sup_{\varphi\in\rew^\star}\{-z(\varphi)\vee\ell\}\le-\ell$.
Finally, recalling that
$\limsup_{t\uparrow\infty}(1/t)\ln\prob[S_1>t]=:\ells\le\ell$ we
realize that for any fixed number $\rho>0$ there exists a positive
constant $M<+\infty$ such that $\prob[S_1>t]\le Me^{(\ell+\rho)t}$ for
all $t\ge 0$. By making use of this bound in (\ref{start_free_sec}) as
well as bounds (\ref{setD_2}) and (\ref{setD_234}) we find
\begin{equation*}
  \nu_t(C)\le M\sum_{\tau=0}^t e^{-\lambda t-\ell(t-\tau)}\cdot e^{(\ell+\rho)(t-\tau)}\le M(t+1)e^{-\lambda t+\rho t}.
\end{equation*}
Thus $\limsup_{t\uparrow\infty}(1/t)\ln \nu_t(C)\le-\lambda+\rho$,
which proves the lemma once $\lambda$ is sent to $\inf_{w\in
  C}\sup_{\varphi\in\rew^\star}\{\varphi(w)-z(\varphi)\vee\ell\}$
and $\rho$ is sent to zero.
\end{proof}

We conclude the section demonstrating an upper large deviation bound
for closed sets under the hypothesis that $\rew$ is finite-dimensional
and that an exponential moment condition holds. No restriction on
$\ells$ is needed here.
\begin{proposition}
  \label{free_closed}
Assume that $\rew$ has finite dimension and that there exist numbers
$\zeta\ge 0$ and $\xi>0$ such that $\Ex[e^{\xi\|X_1\|+v(S_1)-\zeta
    S_1}\mathds{1}_{\{S_1<\infty\}}]<+\infty$. Then, for each
$F\subseteq\rew$ closed
\begin{equation*}
  \limsup_{t\uparrow\infty}\frac{1}{t}\ln\nu_t(F)\le-\adjustlimits\inf_{w\in F}\sup_{\varphi\in\rew^\star}\Big\{\varphi(w)
  -z(\varphi)\vee\ells\Big\}.
\end{equation*}
\end{proposition}

\begin{proof}
  Fix a closed set $F$ in $\rew$ and observe that $\inf_{w\in
    F}\sup_{\varphi\in\rew^\star}\{\varphi(w)-z(\varphi)\vee\ells\}\ge-z(0)\vee\ells>-\infty$. Pick
  a real number $\lambda<\inf_{w\in
    F}\sup_{\varphi\in\rew^\star}\{\varphi(w)
  -z(\varphi)\vee\ells\}$. Let $d$ be the dimension of $\rew$ and
  let $\varphi_1,\ldots,\varphi_{2d}$ be the linear functionals
  introduced in the proof of proposition
  \ref{closed_constrained}. Since $z(\varphi)<+\infty$ if
  $\|\varphi\|\le\xi$ by hypothesis, as we have seen in Section
  \ref{par:th1pp}, there exists a positive number $M<+\infty$ with the
  property that $z(\xi\varphi_i)\le M$ for each $i$. Pick a number
  $\rho>0$ such that $M-\xi\rho\le -\lambda$.  Denoting by $K$ the
  compact set $\cap_{i=1}^{2d}\{w\in\rew:\varphi_i(w)\le\rho\}$ we
  have $K^c=\cup_{i=1}^{2d}\{w\in\rew:\varphi_i(w)> \rho\}$.  This
  way, starting from (\ref{start_free_sec}) and noticing that
  $0\notin\{w\in\rew:\varphi_i(w)> \rho\}$ for all $i$, by using the
  Chernoff bound first and the bound
  $\Ex[U_te^{\xi\varphi(W_t)+H_t}]\le e^{z(\xi\varphi_i)t}$ due to
  lemma \ref{zprop} later we obtain
\begin{eqnarray}
  \nonumber
  \nu_t(K^c)&\le&\sum_{i=1}^{2d} \nu_t\Big(\big\{w\in\rew:\varphi_i(w)>\rho\big\}\Big)\\
  \nonumber
  &=&\sum_{i=1}^{2d}\sum_{\tau=1}^t\Ex\Big[\mathds{1}_{\{\varphi_i(W_\tau)>\rho t\}}U_\tau e^{H_\tau}\Big]\cdot\prob[S_1>t-\tau]\\
  \nonumber
  &\le&\sum_{i=1}^{2d}\sum_{\tau=1}^t\Ex\Big[U_\tau e^{\xi\varphi_i(W_\tau)-\xi\rho t+H_\tau}\Big]
  \le\sum_{i=1}^{2d} \sum_{\tau=1}^t e^{z(\xi\varphi_i) \tau-\xi\rho t}\\
  \nonumber
  &\le&2dte^{Mt-\xi\rho t}\le 2dte^{-\lambda t},
\end{eqnarray}
which gives $\nu_t(F)=\nu_t(F\cap K)+\nu_t(F\cap K^c)\le\nu_t(F\cap
K)+2dte^{-\lambda t}$ for each $t$. On the other hand, proposition
\ref{free_compact} with the compact set $F\cap K$ shows that
$\limsup_{t\uparrow\infty}(1/t)\ln\nu_t(F\cap K)\le-\inf_{w\in F\cap
  K}\sup_{\varphi\in\rew^\star}\{\varphi(w)
-z(\varphi)\vee\ells\}\le -\inf_{w\in
  F}\sup_{\varphi\in\rew^\star}\{\varphi(w)
-z(\varphi)\vee\ells\}\le-\lambda$. Thus,
$\limsup_{t\uparrow\infty}(1/t)\ln\nu_t(F)\le-\lambda$ and the
proposition is proved by sending $\lambda$ to $\inf_{w\in
  F}\sup_{\varphi\in\rew^\star}\{\varphi(w)-z(\varphi)\vee\ells\}$.
\end{proof}

\subsection{Theorem \ref{mainth2} point by point and counterexamples}
\label{th2pp}

Now we explicitly verify theorem \ref{mainth2} point by point.  Assume
$\ells=-\infty$. Then, $\elli=-\infty$ and starting from the facts
that $\ln\prob_t[W_t/t\in\cdot\,]=\ln\nu_t-\ln Z_t$ and
$Z_t:=\Ex[e^{H_t}]$ for all $t\ge 1$ we get part (a) of theorem
\ref{mainth2} thanks to proposition \ref{free_open_nonorm} and formula
(\ref{Zt_sup}) with $\varphi=0$. Similarly, part (b) of theorem
\ref{mainth2} for compact and closed sets is obtained by combining
propositions \ref{free_compact} and \ref{free_closed} with formula
(\ref{Zt_inf}). As far as convex sets is concerned, we observe that
$z(0)-I(0)=-\sup_{\varphi\in\rew^\star}\{-z(\varphi)\}=\inf_{\varphi\in\rew^\star}\{z(\varphi)\}$
so that $z(\varphi)\ge z(0)-I(0)$ for all $\varphi\in\rew^\star$. This
way, part (b) of theorem \ref{mainth2} for convex sets follows when
$I(0)<+\infty$ by invoking lemma \ref{convex_free} with
$\ell:=z(0)-I(0)$ and, again, formula (\ref{Zt_inf}) with $\varphi=0$.

The upper large deviation bound for convex sets cannot hold in general
when $\ells=-\infty$ and $I(0)=+\infty$. We show two examples where it
fails, involving an open convex set and a closed convex set,
respectively. In these examples we assume $\prob[1<S_1<\infty]=1$ and
$v=0$, so that $H_t=0$, $Z_t=1$, and $\prob_t[W_t/t\in\cdot\,]=\nu_t$
for every $t$. The following is the counterexample with an open convex
set.
\begin{example}
Consider the reward $X_i:=S_i$ for each $i$. In this example we have
$\rew=\Rl$, so that for any $\varphi\in\rew^\star$ there exists one
and only one real number $k$ such that $\varphi(w)=kw$ for all $w$.
As $\prob[S_1<\infty]=1$ and $v=0$, by identifying $\varphi$ with $k$
definitions (\ref{defzetafun}) and (\ref{defIorate}) give
$z(k)=\inf\{\zeta\in\Rl:\Ex[e^{kS_1-\zeta S_1}]\le 1\}=k$ for all
$k\in\Rl$, $I(1)=0$, and $I(w)=+\infty$ for each
$w\in\Rl\setminus\{1\}$. The rate function $I$ is consistent with the
fact that $\sum_{i\ge 1}S_i\mathds{1}_{\{T_i\le t\}}=t$ if a renewal
occurs at time $t$. The upper bound
$\limsup_{t\uparrow\infty}(1/t)\ln\prob_t[W_t/t\in C]\le-\inf_{w\in
  C}\{I(w)\}$ does not hold with the open convex set $C:=(-\infty,1)$,
for which $\inf_{w\in C}\{I(w)\}=+\infty$.  Indeed, keeping only the
term corresponding to $\tau=t-1$ in the r.h.s.\ of
(\ref{start_free_sec}), observing that $W_{t-1}/t=1-1/t\in C$ if
$U_{t-1}=1$, and recalling that $\prob[S_1>1]=1$ by assumption, we
find for each $t\ge 2$
\begin{eqnarray}
\nonumber
  1&\ge& \prob_t\bigg[\frac{W_t}{t}\in C\bigg]=\nu_t(C)\\
  \nonumber
  &\ge&\Ex\bigg[\mathds{1}_{\big\{\frac{W_{t-1}}{t}\in C\big\}}U_{t-1}e^{H_{t-1}}\bigg]\cdot\prob[S_1>1]
  =\Ex\big[U_{t-1}e^{H_{t-1}}\big],
\end{eqnarray}
giving $\lim_{t\uparrow\infty}(1/t)\ln\prob_t[W_t/t\in C]=0$ by lemma
\ref{zprop} as $z(0)=0$. 
\end{example}

The following is the counterexample with a closed convex set.
\begin{example}
Consider the reward $X_i:=(S_i,Y_i)$ for each $i$ with $Y_i$
independent of $S_i$ and distributed according to the standard Cauchy
law: $\prob[Y_i\le y]=(1/\pi)[\pi/2+\arctan(y)]$ for all $y\in\Rl$. In
this example $\rew=\Rl^2$, so that for any $\varphi\in\rew^\star$
there exists one and only one pair of real numbers $k=(k_S,k_Y)$ such
that $\varphi(w)=k_Sw_S+k_Yw_Y$ for all $w=(w_S,w_Y)$.  As
$\prob[S_1<\infty]=1$ and $v=0$, and as $Y_1$ has no exponential
moments, by identifying $\varphi$ with $k$ definition
(\ref{defzetafun}) gives $z(k)=\inf\{\zeta\in\Rl:\Ex[e^{k_S S_1-\zeta
    S_1}]\cdot\Ex[e^{k_Y Y_1}]\le 1\}=k_S$ if $k_Y=0$ and
$z(k)=+\infty$ if $k_Y\ne 0$. It follows from definition
(\ref{defIorate}) that $I(w)=0$ if $w_S=1$ and $I(w)=+\infty$
otherwise. The upper bound
$\limsup_{t\uparrow\infty}(1/t)\ln\prob_t[W_t/t\in C]\le-\inf_{w\in
  C}\{I(w)\}$ does not hold with the closed convex set
$C:=\{w\in\Rl^2:w_S<1\mbox{ and }w_Y\ge1/(1-w_S)\}$, for which
$\inf_{w\in C}\{I(w)\}=+\infty$. Indeed, as we shall show in a moment
we have for every $t\ge 2$
\begin{equation}
  1\ge \prob_t\bigg[\frac{W_t}{t}\in C\bigg]=\nu_t(C)\ge\prob\big[Y_1\ge t^2\big]\cdot\Ex\big[U_{t-1}e^{H_{t-1}}\big],
\label{toprove}
\end{equation}
giving $\lim_{t\uparrow\infty}(1/t)\ln\prob_t[W_t/t\in C]=0$ by lemma
\ref{zprop} as $z(0)=0$.

In order to prove (\ref{toprove}) we pick an integer $t\ge 2$ and
observe that when a renewal occurs at the time $t-1$, so that
$\sum_{i\ge 1}S_i\mathds{1}_{\{T_i\le t-1\}}=t-1$, then $W_{t-1}/t\in
C$ if and only if $\sum_{i\ge 1}Y_i\mathds{1}_{\{T_i\le t-1\}}\ge
t^2$. This way, keeping only the term corresponding to $\tau=t-1$ in
the r.h.s.\ of (\ref{start_free_sec}) and recalling that
$\prob[S_1>1]=1$ we get 
\begin{eqnarray}
\nonumber
\nu_t(C)&\ge&\Ex\bigg[\mathds{1}_{\big\{\frac{W_{t-1}}{t}\in C\big\}}U_{t-1}e^{H_{t-1}}\bigg]
\cdot\prob[S_1>1]\\
\nonumber
&=&\Ex\bigg[\mathds{1}_{\big\{\sum_{i\ge 1}Y_i\mathds{1}_{\{T_i\le t-1\}}\ge t^2\big\}}U_{t-1}e^{H_{t-1}}\bigg]\\
\nonumber
&=&\sum_{n=1}^{t-1}\Ex\bigg[\mathds{1}_{\big\{\sum_{i=1}^nY_i\ge t^2\big\}}\mathds{1}_{\{T_n=t-1\}}e^{H_{t-1}}\bigg]\\
\nonumber
&=&\sum_{n=1}^{t-1}\prob\Bigg[\sum_{i=1}^nY_i\ge t^2\Bigg]\cdot\Ex\Big[\mathds{1}_{\{T_n=t-1\}}e^{H_{t-1}}\Big].
\end{eqnarray}
On the other hand, $(1/n)\sum_{i=1}^nY_i$ is distributed as $Y_1$ by
the stability property of the Cauchy law so that
\begin{eqnarray}
\nonumber
  \nu_t(C)&\ge&\sum_{n=1}^{t-1}\prob\Bigg[\sum_{i=1}^nY_i\ge t^2\Bigg]\cdot\Ex\Big[\mathds{1}_{\{T_n=t-1\}}e^{H_{t-1}}\Big]\\
\nonumber
&=&\sum_{n=1}^{t-1}\prob\Big[nY_1\ge t^2\Big]\cdot\Ex\Big[\mathds{1}_{\{T_n=t-1\}}e^{H_{t-1}}\Big]\\
\nonumber
&\ge&\sum_{n=1}^{t-1}\prob\Big[Y_1\ge t^2\Big]\cdot\Ex\Big[\mathds{1}_{\{T_n=t-1\}}e^{H_{t-1}}\Big]\\
\nonumber
&=&\prob\big[Y_1\ge t^2\big]\cdot\Ex\big[U_{t-1}e^{H_{t-1}}\big].
\end{eqnarray}
\end{example}

\subsection{Theorem \ref{mainth3} point by point}
\label{th3pp}

To conclude, we explicitly verify theorem \ref{mainth3} point by
point. Assume $\ells>-\infty$. The functions $\Ii$ and $\Is$ defined
by (\ref{def:Ibasso}) and (\ref{def:Ialto}) are the Fenchel-Legendre
transform of $z\vee\elli-z(0)\vee\ells$ and
$z\vee\ells-z(0)\vee\elli$, respectively. Convexity and lower
semicontinuity of $\Ii$ and $\Is$ are immediate to check. The
functions $\Ii$ and $\Is$ are proper convex. Indeed, considering for
instance $\Ii$, we have on the one hand
$\Ii(w)\ge-z(0)\vee\elli+z(0)\vee\ells>-\infty$ for all $w\in\rew$,
and on the other hand $\Ii(u)\le J(u)+z(0)\vee\ells<+\infty$ at some
point $u$ because $J$ is proper convex. These arguments demonstrate
part (a) of theorem \ref{mainth3}. As far as part (b) and part (c) is
concerned, we recall that $\ln\prob_t[W_t/t\in\cdot\,]=\ln\nu_t-\ln
Z_t$ and that $Z_t:=\Ex[e^{H_t}]$ for all $t$ in such a way that part
(b) follows from proposition \ref{free_open_nonorm} and formula
(\ref{Zt_sup}) with $\varphi=0$. Part (c) for compact and closed sets
is due to propositions \ref{free_compact} and \ref{free_closed}
combined with formula (\ref{Zt_inf}). Finally, part (c) for convex
sets follows from lemma \ref{convex_free} with $\ell=\ells$ and,
again, formula (\ref{Zt_inf}).

\section*{Acknowledgements}
The author is grateful to Paolo Tilli for useful discussions about the
counterexamples presented in Section\ \ref{th2pp} and to Francesco
Caravenna and Paolo Dai Pra for valuable overall comments.

\appendix

\section{Proof of lemma \ref{confrontoIL}}
\label{proof:confrontoIL}

Since $S_1<\infty$ with full probability and $v=0$, according to
definition (\ref{defzetafun}) we have
$z(\varphi)=\inf\{\zeta\in\Rl\,:\,\Ex[e^{\varphi(X_1)-\zeta S_1}]\le
1\}$ for all $\varphi\in\rew^\star$. We shall show that
for every $\beta\ge 0$ and $w\in\rew$
\begin{equation}
  \label{confrontoIL_1}
\Upsilon(\beta,w)=\sup_{\varphi\in\scriptsize{\dom}z}\Big\{\varphi(w)-\beta z(\varphi)\Big\}, 
\end{equation}
where $\dom z:=\{\varphi\in\rew^\star:z(\varphi)<+\infty\}$ is the
effective domain of the function $z$. The identity $I=\Lambda$
immediately follows from (\ref{confrontoIL_1}) by taking $\beta=1$ and
proves part (a) of the lemma. Regarding part (b), assume
$\ells>-\infty$ and consider the function $F$ that for a given
$w\in\rew$ maps $(\beta,\varphi)\in[0,1]\times\dom z$ in the real
number $F(\beta,\varphi):=\varphi(w)-\beta z(\varphi)-(1-\beta)\ells$.
The function $F$ is concave and upper semicontinuous with respect to
$\varphi$ for each fixed $\beta\in[0,1]$, inheriting these properties
from $z$, and convex and continuous with respect to $\beta$ for each
fixed $\varphi\in\dom z$. Then, due to compactness of the closed
interval $[0,1]$, Sion's minimax theorem allows us to exchange the
infimum over $\beta\in[0,1]$ and the supremum over $\varphi\in\dom z$:
$\sup_{\varphi\in\scriptsize{\dom}z}\inf_{\beta\in[0,1]}\big\{F(\beta,\varphi)\big\}=
\inf_{\beta\in[0,1]}\sup_{\varphi\in\scriptsize{\dom}z}\big\{F(\beta,\varphi)\big\}$.
Since $z(0)=0$ and $z(\varphi)=+\infty$ if $\varphi\not\in\dom z$,
this identity yields
\begin{eqnarray}
  \nonumber
  \Is(w)&=&\sup_{\varphi\in\rew^\star}\Big\{\varphi(w)-z(\varphi)\vee\ells\Big\}=
  \adjustlimits\sup_{\varphi\in\scriptsize{\dom}z}\inf_{\beta\in[0,1]}\big\{F(\beta,\varphi)\big\}\\
  \nonumber
  &=&\adjustlimits\inf_{\beta\in[0,1]}\sup_{\varphi\in\scriptsize{\dom}z}\big\{F(\beta,\varphi)\big\}=
  \adjustlimits\inf_{\beta\in[0,1]}\sup_{\varphi\in\scriptsize{\dom}z}\Big\{\varphi(w)-\beta z(\varphi)-(1-\beta)\ells\Big\}\\
  \nonumber
  &=&\inf_{\beta\in[0,1]}\Big\{\Upsilon(\beta,w)-(1-\beta)\ells\Big\}=:\Lambdas(w).
\end{eqnarray}
This way, part (b) of the lemma is demonstrated as $w$ is an arbitrary
point.

Let us prove (\ref{confrontoIL_1}). Pick $\beta\ge 0$ and
$w\in\rew$. To begin with, we point out that the function that
associates $\zeta$ with $\Ex[e^{\varphi(X_1)-\zeta S_1}]$ for a given
$\varphi\in\rew^\star$ is lower semicontinuous by Fatou's lemma, so
that $\Ex[e^{\varphi(X_1)-z(\varphi) S_1}]\le 1$ if
$z(\varphi)<+\infty$. It follows that for any $\varphi\in\dom z$
\begin{eqnarray}
  \nonumber
  \LC(\beta,w)&:=&\sup_{(\zeta,\vartheta)\in\Rl\times\rew^\star}\Big\{\vartheta(w)-\beta\zeta-\ln\Ex\big[e^{\vartheta(X_1)-\zeta S_1}\big]\Big\}\\
  \nonumber
  &\ge&\varphi(w)-\beta z(\varphi)-\ln\Ex\big[e^{\varphi(X_1)-z(\varphi) S_1}\big]\ge\varphi(w)-\beta z(\varphi),
\end{eqnarray}
so that
$\inf_{\gamma>0}\{\gamma\LC(\beta/\gamma,w/\gamma)\}\ge\varphi(w)-\beta
z(\varphi)$. Continuity of $\varphi$ results in
$\Upsilon(\beta,w)\ge\varphi(w)-\beta z(\varphi)$ and the
arbitrariness of $\varphi$ gives
\begin{equation*}
\Upsilon(\beta,w)\ge\sup_{\varphi\in\scriptsize{\dom}z}\Big\{\varphi(w)-\beta z(\varphi)\Big\}.
\end{equation*}

The opposite bound, which leads us to the proof of
(\ref{confrontoIL_1}), is more involved and is achieved through the
following two inequalities:
\begin{eqnarray}
\label{confrontoIL_2}
\Upsilon(\beta,w)&\le&\adjustlimits\inf_{\gamma\in[0,\beta]}\sup_{(\zeta,\varphi)\in\mathcal{D}}\Big\{\varphi(w)-\beta\zeta-
\gamma\ln\Ex\big[e^{\varphi(X_1)-\zeta S_1}\big]\Big\}\\
\label{confrontoIL_4}
&\le&\sup_{\varphi\in\scriptsize{\dom}z}\Big\{\varphi(w)-\beta z(\varphi)\Big\},
\end{eqnarray}
where
$\mathcal{D}:=\{(\zeta,\varphi)\in\Rl\times\rew^\star:\Ex[e^{\varphi(X_1)-\zeta
    S_1}]<+\infty\}$. To get at (\ref{confrontoIL_2}) we observe that
the definition of $\Upsilon$ immediately gives
\begin{equation}
\label{confrontoIL_3}
\Upsilon(\beta,w)\le\sup_{(\zeta,\varphi)\in\mathcal{D}}\Big\{\varphi(w)-\beta\zeta-
\gamma\ln\Ex\big[e^{\varphi(X_1)-\zeta S_1}\big]\Big\}
\end{equation}
for all $\gamma>0$. The lower-semicontinuous regularization procedure
used to construct $\Upsilon$ entails that this bound also holds for
$\gamma=0$, as we now show, thus giving (\ref{confrontoIL_2}). The
rate function $\LC$ is proper convex by Cram\'er's theorem, so that
there exist $\beta_o\in\Rl$ and $w_o\in\rew$ such that
$\LC(\beta_o,w_o)$ is finite. It follows that
$\varphi(w_o)-\beta_o\zeta-\ln\Ex[e^{\varphi(X_1)-\zeta S_1}]\le
\LC(\beta_o,w_o)=:c\ge 0$, that is $\ln\Ex[e^{\varphi(X_1)-\zeta
    S_1}]\ge\varphi(w_o)-\beta_o\zeta-c$ for all $\zeta\in\Rl$ and
$\varphi\in\rew^\star$. Then, for every $\delta>0$ and
$\gamma_o\in(0,\delta)$ such that $\gamma_o|\beta_o|<\delta$ and
$\gamma_o\|w_o\|<\delta$ we find
\begin{eqnarray}
  \nonumber
  \adjustlimits\inf_{\alpha\in (\beta-\delta,\beta+\delta)}\inf_{u\in B_{w,\delta}}\inf_{\gamma>0}
  \big\{\gamma\LC(\alpha/\gamma,u/\gamma)\big\}&\le&
  \inf_{\gamma>0}\big\{\gamma\LC(\beta/\gamma+\gamma_o\beta_o/\gamma,w/\gamma+\gamma_ow_o/\gamma)\big\}\\
  \nonumber
  &\le&\gamma_o\LC(\beta/\gamma_o+\beta_o,w/\gamma_o+w_o)\\
  \nonumber
  &=&\sup_{(\varphi,\zeta)\in\mathcal{D}}\Big\{\varphi(w+\gamma_ow_o)-(\beta+\gamma_o\beta_o)\zeta-
  \gamma_o\ln\Ex\big[e^{\varphi(X_1)-\zeta S_1}\big]\Big\}\\
  \nonumber
  &\le&\sup_{(\varphi,\zeta)\in\mathcal{D}}\Big\{\varphi(w)-\beta\zeta\Big\}+\gamma_o c\\
  \nonumber
  &\le&\sup_{(\varphi,\zeta)\in\mathcal{D}}\Big\{\varphi(w)-\beta\zeta\Big\}+\delta c.
\end{eqnarray}
It follows that
\begin{eqnarray}
  \nonumber
  \Upsilon(\beta,w)&:=&\adjustlimits\lim_{\delta\downarrow 0}\inf_{\alpha\in (\beta-\delta,\beta+\delta)}\inf_{u\in B_{w,\delta}}\inf_{\gamma>0}
  \big\{\gamma\LC(\alpha/\gamma,u/\gamma)\big\}\\
  \nonumber
  &\le& \sup_{(\varphi,\zeta)\in\mathcal{D}}\Big\{\varphi(w)-\beta\zeta\Big\},
\end{eqnarray}
which exactly is (\ref{confrontoIL_3}) when $\gamma=0$.

Let us move to bound (\ref{confrontoIL_4}), which we prove by invoking
Sion's minimax theorem once again.  The function that associates
$(\zeta,\varphi)\in\mathcal{D}$ with $\Ex[e^{\varphi(X_1)-\zeta S_1}]$
is lower semicontinuous by Fatou's lemma and convex, so that the real
function that maps
$(\gamma,\zeta,\varphi)\in[0,\beta]\times\mathcal{D}$ in
$\varphi(w)-\beta\zeta-\gamma\ln\Ex[e^{\varphi(X_1)-\zeta S_1}]$ is
concave and upper semicontinuous with respect to $(\zeta,\varphi)$ for
each fixed $\gamma\in[0,\beta]$ and convex and continuous with respect
to $\gamma$ for each fixed pair $(\zeta,\varphi)\in\mathcal{D}$. Then,
Sion's theorem ensures us that
\begin{eqnarray}
  \nonumber
  &&\adjustlimits\inf_{\gamma\in[0,\beta]}\sup_{(\zeta,\varphi)\in\mathcal{D}}\Big\{\varphi(w)-\beta\zeta-
  \gamma\ln\Ex\big[e^{\varphi(X_1)-\zeta S_1}\big]\Big\}\\
  \nonumber
  &=&\adjustlimits\sup_{(\zeta,\varphi)\in\mathcal{D}}\inf_{\gamma\in[0,\beta]}\Big\{\varphi(w)-\beta\zeta-
  \gamma\ln\Ex\big[e^{\varphi(X_1)-\zeta S_1}\big]\Big\}\\
  \nonumber
  &=&\sup_{(\zeta,\varphi)\in\mathcal{D}}\Big\{\varphi(w)-\beta\zeta-\beta\ln 1\vee\Ex\big[e^{\varphi(X_1)-\zeta S_1}\big]\Big\}.
\end{eqnarray}
On the other hand, if $(\zeta,\varphi)\in\mathcal{D}$, then
$\varphi\in\dom z$ because $\Ex[e^{\varphi(X_1)-(\zeta+h) S_1}]\le
e^{-h}\Ex[e^{\varphi(X_1)-\zeta S_1}]\le 1$ for all sufficiently large
$h$ as $S_1\ge 1$ with full probability. It follows that
\begin{eqnarray}
  \nonumber
  &&\adjustlimits\inf_{\gamma\in[0,\beta]}\sup_{(\zeta,\varphi)\in\mathcal{D}}\Big\{\varphi(w)-\beta\zeta-
  \gamma\ln\Ex\big[e^{\varphi(X_1)-\zeta S_1}\big]\Big\}\\
  \nonumber
  &=&\sup_{(\zeta,\varphi)\in\mathcal{D}}\Big\{\varphi(w)-\beta\zeta-\beta\ln 1\vee\Ex\big[e^{\varphi(X_1)-\zeta S_1}\big]\Big\}\\
  \nonumber
  &\le&\sup_{\varphi\in\scriptsize{\dom}z}\sup_{\zeta\in\Rl}\Big\{\varphi(w)-\beta\zeta-\beta\ln 1\vee\Ex\big[e^{\varphi(X_1)-\zeta S_1}\big]\Big\}\\
  \nonumber
   &\le&\sup_{\varphi\in\scriptsize{\dom}z}\Big\{\varphi(w)-\beta z(\varphi)\Big\},
\end{eqnarray}
where the last bound demonstrates (\ref{confrontoIL_4}) and is
justified as follows.  Pick any $\varphi\in\dom z$ and recall that
$\Ex[e^{\varphi(X_1)-\zeta S_1}]\le 1$ or $\Ex[e^{\varphi(X_1)-\zeta
    S_1}]>1$ depending on whether $\zeta\ge z(\varphi)$ or
$\zeta<z(\varphi)$ by definition of $z(\varphi)$. The function that
associates $\zeta$ with $\zeta+\ln 1\vee\Ex[e^{\varphi(X_1)-\zeta
    S_1}]$ is lower semicontinuous and increasing for $\zeta\ge
z(\varphi)$. It is not increasing for $\zeta<z(\varphi)$ because
$\zeta+\ln 1\vee\Ex[e^{\varphi(X_1)-\zeta
    S_1}]=\ln\Ex[e^{\varphi(X_1)-\zeta(S_1-1)}]$ in this case and
$S_1\ge 1$ with full probability. Then, this function attains a global
minimum at $\zeta=z(\varphi)$, so that $\zeta+\ln
1\vee\Ex[e^{\varphi(X_1)-\zeta S_1}]\ge z(\varphi)$ for all
$\zeta\in\Rl$.



\end{document}